\documentclass{eajam}
\renewcommand\thisnumber{x}
\renewcommand\thisyear {200x}
\renewcommand\thismonth{xxx}
\renewcommand\thisvolume{xx}

\usepackage{graphicx}
\usepackage{subfigure}
\usepackage{cleveref}
\usepackage{amsmath,amsthm,amssymb}
\usepackage{tikz}
\usepackage{multirow}

\crefname{equation}{}{}
\crefname{theorem}{Theorem}{Theorems}
\crefname{figure}{Figure}{Figures}
\crefname{table}{Table}{Tables}
\crefname{section}{Section}{Sections}
\crefname{example}{Example}{Examples}
\crefname{lemma}{Lemma}{Lemmas}
\crefname{remark}{Remark}{Remarks}

\begin{document}

\markboth{X.Z. Yang and X.B. Yin}{Analysis of Convergence for the IPA-AC Method}
\title{Analysis of Convergence for the IPA-AC Method}


\author[F.~Author and A.~Co-Author(s)]{Xiuzhu Yang\affil{1} and Xiaobo Yin\affil{2}\comma\corrauth}
\address{\affilnum{1}\ School of Mathematics and Statistics, and Hubei Key Laboratory of Mathematical Sciences, Central China Normal University, Wuhan, China. \\
\affilnum{2}\ School of Mathematics and Statistics, and Key Laboratory of Nonlinear Analysis \& Applications (Ministry of Education), Central China Normal University, Wuhan, China.}
%
%
\emails{{\tt janyyang@mails.ccnu.edu.cn} (X.Z. Yang), {\tt yinxb@ccnu.edu.cn} (X.B. Yin)}
%
\begin{abstract}
The Improved Partial Area-Analytical Calculation (IPA-AC) method represents a leading meshfree discretization strategy for peridynamic models, distinguished by its rigorous geometric treatment of boundary intersections via dual corrections of integration weights and quadrature points. Despite its empirical success in suppressing boundary-induced geometric errors, a systematic theoretical characterization of its convergence behaviors under distinct scaling limits has remained elusive. This work establishes a unified convergence framework for the IPA-AC method applied to both scalar and tensor kernels. By leveraging the Lax Equivalence Theorem, we explicitly derive error estimates that reveal the method's performance across three critical limiting regimes. The theoretical analysis, substantiated by numerical validation, demonstrates that: (1) for a fixed horizon $\delta$, the method achieves robust second-order convergence $\mathcal{O}(h ^{2})$ with respect to the mesh size $h$; (2) for a fixed mesh, the discretization error scales as $\mathcal{O}(\delta^{-2})$, indicating a sensitivity to the nonlocal length scale; and (3) the method does not satisfy the Asymptotic Compatibility (AC) condition. These findings clarify that while the IPA-AC method offers superior accuracy for simulating fixed nonlocal models, it requires a sufficiently large horizon-to-mesh ratio to mitigate intrinsic discretization errors when approximating the local limit.
\end{abstract}



\maketitle


\section{Introduction}\label{sec:Introduction}
Peridynamics, originally introduced by Silling \cite{silling2000,silling2007peridynamic}, has become an important framework for modeling materials containing cracks and other displacement discontinuities \cite{javili2019peridynamics}. In contrast to classical continuum mechanics, whose governing equations rely on spatial derivatives, peridynamics adopts integral operators and thereby circumvents the mathematical difficulties associated with discontinuities. This nonlocal formulation enables the model to naturally accommodate jumps in the displacement field and has proven particularly effective for describing crack nucleation and propagation.

To facilitate numerical solutions in engineering contexts, suitable discretizations of the nonlocal integral operators are essential. The most natural strategy is to approximate the integral via a finite summation. Silling proposed the first meshfree discretization of this type \cite{silling2005meshfree}, in which the reference configuration is discretized into a set of nodes, each associated with a specific finite volume (or area) defined as a cell. The properties of a node reflect the average physical characteristics of its corresponding cell, yielding a piecewise-constant approximation. Moreover, since this approach does not rely on any background mesh, peridynamics can be viewed as a coarse-grained extension of molecular dynamics (MD) \cite{seleson2009peridynamics,seleson2014peridynamic,askari2008peridynamics}. Leveraging this conceptual similarity, the method can be straightforwardly implemented within existing MD-type simulation frameworks \cite{parks2011peridynamics,parks2008implementing}. Beyond meshfree schemes, alternative discretizations based on finite elements, collocation methods, and other numerical formulations have also been proposed and systematically studied \cite{ren2022fem,chen2011continuous,emmrich2007peridynamic,kilic2009structural,yu2011new}.

Regardless of the numerical method used, a central issue in the numerical treatment of nonlocal models is the convergence behavior of discrete approximations, which underpins both physical fidelity and numerical reliability. Bobaru et al. \cite{bobaru2009convergence} identified three representative limiting regimes for assessing convergence:
\begin{enumerate}
	\item ($\delta m$)-convergence: Occurs when $\delta\to0$ and $m$ increases accordingly, with the growth rate of m outpacing the decay rate of $\delta$. Under this regime, the numerical approximation not only converges to the analytical nonlocal solution, but also converges almost everywhere uniformly to the local classical solution.
	\item $m$-convergence: Occurs when the horizon scale δ is fixed and $m\to\infty$ (equivalent to $h\to0$). In this limit, the numerical peridynamic approximation converges to the exact nonlocal peridynamic solution for the given fixed $\delta$.
	\item $\delta$-convergence: Occurs when $\delta\to0$ while $m$ is kept as a fixed constant (i.e., $\delta$ and $h$ decrease proportionally), or when the growth rate of $m$ is slower than the decay rate of $\delta$. In this case, the numerical approximation only converges almost everywhere to a certain approximation of the classical solution, and such convergence cannot be guaranteed to be uniform in theory.
	\end{enumerate}

	The pioneering work of Bobaru et al. keenly pointed out that when the nonlocal scale and the mesh scale tend to zero proportionally, i.e., under the aforementioned $\delta$-convergence regime, the discrete model may suffer from inconsistent or indeterminate convergence issues in its transition to the classical local theory. It is precisely to rigorously address and resolve this theoretical pitfall that the rigorous mathematical foundation of asymptotic compatibility has been gradually established and systematically developed in the computational mathematics community. In this process, the first and foremost task is to construct a rigorous mathematical space and variational framework for nonlocal models and their numerical discretization. Du and co-authors provided a rigorous analytical framework for nonlocal diffusion problems on bounded domains in \cite{Du2012analysis}, and theoretically proved that this framework enables the investigation of finite-dimensional approximations using both discontinuous and continuous Galerkin methods. Subsequently, Du and co-authors systematically developed the theory of nonlocal vector calculus in 2013 \cite{du2013nonlocal}. They defined nonlocal divergence, gradient and curl operators, derived the corresponding adjoint operators, and established the nonlocal counterparts of several fundamental theorems and identities in classical differential vector calculus, providing indispensable mathematical tools for the subsequent development of nonlocal theory.
	
	Building on this solid theoretical foundation, Tian and Du further investigated the dynamic error propagation mechanism under the dual limits of $\delta\to0$ and $h\to0$. In \cite{Tian2014asymptotically}, Tian and Du established a rigorous abstract mathematical framework for such parameterized problems and proposed the asymptotically compatible schemes. Their work theoretically proved that the Galerkin approximation inherently preserves asymptotic compatibility, as long as the finite element discrete space contains continuous piecewise linear functions. This series of seminal achievements not only thoroughly clarifies the potential risks in the discretization of nonlocal models, but also provides a solid theoretical cornerstone and universal guiding principles for dealing with multiscale nonlocal interactions.
	
	For state-based peridynamics with singular kernels, Trask et al. \cite{trask2019asymptotically} proposed the first meshfree quadrature method proven to satisfy AC. These theoretical advancements emphasize that AC property is a prerequisite for the consistency of nonlocal discrete models in the local limit. Given the superior accuracy of the IPA-AC method in practical applications, a rigorous investigation into its asymptotic compatibility is of paramount necessity to characterize its convergence behavior in the local limit.

To contextualize the IPA-AC method within this theoretical framework, it is instructive to first review the evolution of peridynamic quadrature schemes. In early work, Silling and Askari proposed a scheme, later referred to as the full-area (FA) scheme, in which the entire area of a cell is assigned to the interaction domain whenever the centroid of cell lies inside the horizon, regardless of how the horizon actually cuts the cell\cite{silling2005meshfree}. This simple strategy is easy to implement but introduces significant errors for boundary cells. To alleviate this issue, the LAMMPS algorithm in \cite{parks2008implementing} augments the FA scheme with a distance-dependent scaling factor: for cells whose centers lie inside the horizon but are close to its boundary, the effective area is rescaled by this factor, which partially reduces the geometric error; however, cells whose centers lie outside the horizon but still intersect the interaction domain are not accounted for and thus remain uncorrected. Building on these ideas, the QWJ algorithm was developed in \cite{le2014two}, which adopts a more refined volume correction strategy to adjust the effective area of elements near the boundary. This approach also incorporates elements partially outside the interaction domain into the area correction framework, thereby achieving more stable and faster m-convergence. The PA-AC algorithm in \cite{seleson2014improved} exhaustively enumerates all possible intersection scenarios between the circular horizon and square cells, providing analytical formulas for the exact intersection area in each case. This rigorous geometric treatment effectively reduces geometric errors to machine precision. In the same study,
Seleson proposed the IPA-AC method, which further designates the centroid of the intersecting patch as the integration point, establishing itself as the premier single-point meshfree approach. By combining the computational efficiency of single-point integration with dual corrections for both area and integration point location, this method achieves an optimal balance between high performance and superior numerical accuracy. Owing to these advantages, the present study adopts this dual-correction-based algorithm as its central framework for further analysis.

Although \cite{Pablo2014} provided preliminary convergence observations for the IPA-AC method, the analysis focused primarily on mesh refinement with fixed $\delta$ and relied heavily on numerical results. A systematic theoretical characterization of the method's behavior under other essential scaling limits remains absent. In particular, existing work has not rigorously established how the discretization error varies with the horizon $\delta$ for a fixed cell size $h$, nor has it provided a theoretical proof regarding the method's asymptotic compatibility. Moreover, most prior analyses are restricted to one-dimensional settings, and rigorous convergence results for tensor kernel peridynamic models common in engineering applications are largely unavailable. Motivated by these gaps, this work develops a unified theoretical framework for analyzing the convergence of the IPA-AC method.

In this work, we present error estimates for the IPA-AC method using both the linear decay scalar kernel and the linear elastic tensor kernel under the three classes of scaling limits. Our results indicate that:\\
(1)When $\delta$ is fixed, the IPA-AC method exhibits a second-order convergence rate with respect to $h$.\\
(2)When $h$ is fixed, the error increases at a rate of  $\mathcal{O}(\delta^{-2})$  as $\delta$ decreases, demonstrating the influence of the nonlocal length scale.\\
(3)In the concurrent limit where the ratio $\delta/h$ is held constant, the error remains at the $\mathcal{O}(1)$ order and does not decrease as the scale reduces. This implies that the IPA-AC method does not possess the asymptotic compatibility property.\\
These analytical results indicate that, if the ratio between the nonlocal horizon and the cell size is not properly considered, the numerical method may diverge when approximating the local linear elastic theory, thereby theoretically highlighting the importance of the AC framework in the design of nonlocal numerical methods.

The remainder of this paper is organized as follows. \cref{Sec:Predynamic} introduces the governing equations of the bond-based peridynamic model, defining both the scalar and tensor kernels used in this study. \cref{Sec:IPA-ACalgorithm} details the spatial discretization and the implementation of the IPA-AC algorithm, clarifying the geometric dual-correction strategy. \cref{Sec:Convergence} establishes the core theoretical framework, where the consistency and stability of the method are analyzed to derive rigorous error estimates under different scaling limits. \cref{Sec:Numerical} presents comprehensive numerical experiments to validate the theoretical convergence rates and to assess the asymptotic compatibility of the scheme. Finally, concluding remarks and future perspectives are summarized in \cref{Sec:Conclusion}.

\section{Predynamic Model}\label{Sec:Predynamic}
For a continuous body $\mathcal{B}$, the governing equation of motion in peridynamics is given by
\begin{equation*}
\rho\ddot{\mathbf{u}}(\mathbf{x},t)=
\int_{\mathcal{B}}\mathbf{f}\big(\mathbf{u}(\mathbf{x}',t)
-\mathbf{u}(\mathbf{x},t),\mathbf{x}'-\mathbf{x}\big)dV_{\mathbf{x}'}
+\mathbf{b}(\mathbf{x},t).
\end{equation*}
where $\rho$ denotes the mass density, $\mathbf{b}(\mathbf{x},t)$ is the prescribed body force, and $\mathbf{f}$ represents the pairwise force density acting between material points $\mathbf{x},\mathbf{y}\in\mathcal{B}\subset\mathbb{R}^2$.
For a linear microelastic material, the pairwise force assumes a Hookean structure in which the interaction depends linearly on the relative displacement\cite{silling2005meshfree}. In this case,
\begin{equation*}
	\mathbf{f}(\boldsymbol{\xi},\boldsymbol{\eta},t) = \mathbf{C}(\boldsymbol{\xi})\boldsymbol{\eta}(t),
\end{equation*}
where $\mathbf{C}$ is the material's micromodulus function, and
$$
\boldsymbol{\xi} = \mathbf{x}'-\mathbf{x},\qquad \boldsymbol{\eta}(t) = \mathbf{u}(\mathbf{x}',t)-\mathbf{u}(\mathbf{x},t),
$$
denote, respectively, the relative position in the reference configuration and the corresponding relative displacement.The peridynamic equation then reduces to the linear integral form
\begin{equation*}
\rho\ddot{\mathbf{u}}(\mathbf{x},t)
=\int_{\mathcal{B}}\mathbf{C}(\boldsymbol{\xi})
\left(\mathbf{u}(\mathbf{x}',t)-\mathbf{u}(\mathbf{x},t)\right)dV_{\mathbf{x}'}
+\mathbf{b}(\mathbf{x},t).
\end{equation*}

In the peridynamic framework, which is intrinsically nonlocal, a material point $\mathbf{x}$ is assumed to interact with other points within a finite distance $\delta$, a characteristic length scale known as the \emph{horizon}. Consequently, the set of all points within this range constitutes the \emph{neighborhood} of $\mathbf{x}$, defined as: 
\begin{equation*} 
	\mathcal{H}(\mathbf{x},\delta):=
	\{\mathbf{x'}\in\mathcal{B}:|\mathbf{x'}-\mathbf{x}|\leq\delta\}. 
\end{equation*} 
In the nonlocal framework, while the neighborhood $\mathcal{H}(\mathbf{x}_{i},\delta)$ of an interior point $\mathbf{x}_{i}$ resides entirely within the physical body $\mathcal{B}$, the neighborhoods of points in the boundary's vicinity inevitably intersect the exterior.

\begin{figure}
	\centering
	\includegraphics[width=0.7\linewidth]{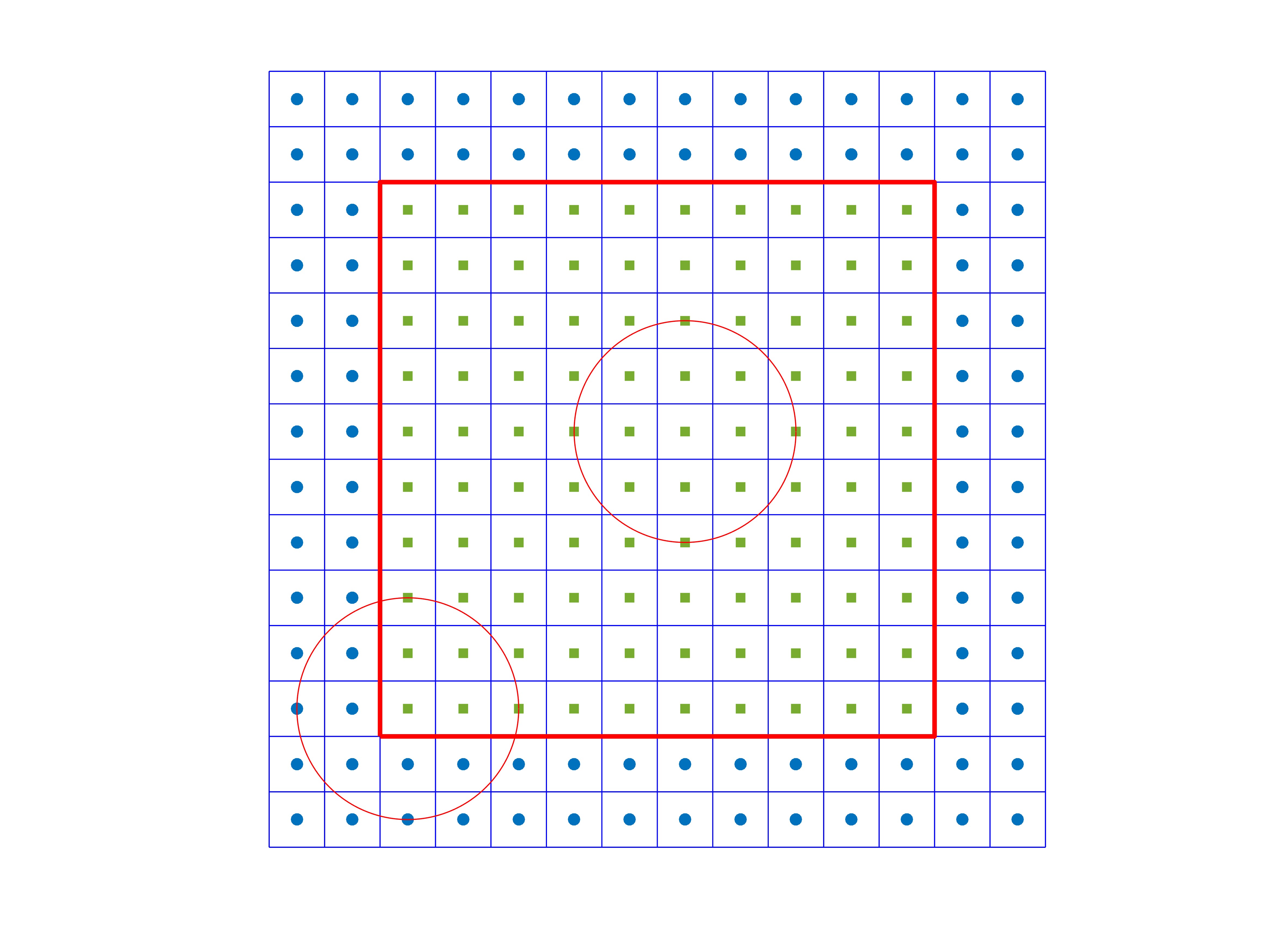}
	\caption{The red border delineates the physical body $\mathcal{B}$ from the surrounding fictitious layer. Note that the neighborhoods of interior nodes reside entirely within $\mathcal{B}$, whereas those of nodes near the boundary partially intersect the fictitious layer.}
	\label{fig:translational_symmetry}
\end{figure}

Since this study primarily focuses on the intrinsic convergence properties of the discrete operator, the physical domain is augmented with a horizon-thick fictitious layer $\hat{\mathcal{B}}$ of width $\delta$ to eliminate complex boundary truncation effects as illustrated in \cref{fig:translational_symmetry}. This extension provides a unified formulation for both interior and boundary-proximate nodes. Dirichlet-type volume constraints are subsequently prescribed over $\hat{\mathcal{B}}$ to facilitate a rigorous error analysis.

Accordingly, for any point $\mathbf{x}$ in the $\mathcal{B}$, the integral governing the pairwise force interactions can be expressed as a nonlocal operator: 
\begin{equation*}
\mathcal{L}_{\delta} \mathbf{u}(\mathbf{x}):= -\int_{\mathcal{H}(\mathbf{x},\delta)}\mathbf{C}(\boldsymbol{\xi})
\left(\mathbf{u}(\mathbf{x}',t)-\mathbf{u}(\mathbf{x},t)\right) dV_{\mathbf{x}'}.
\end{equation*}
Within the nonlocal modeling community, the function $\mathbf{C}(\boldsymbol{\xi})$ is commonly referred to as the \emph{kernel}, which encodes both the range and the intensity of the material interactions.

\subsection{Kernel function}
In this work we examine two classes of kernels that are among the most representative in the peridynamics literature:
(i) scalar kernels, namely compactly supported radial micromoduli with smoothly decaying profiles;
(ii) directional tensor kernels, specifically rank-2 tensor functions encoding component-coupled interactions, hereafter termed tensor kernels.
These two kernel families encode distinct nonlocal interaction structures and modeling assumptions, and it is therefore essential to analyze their convergence behaviors separately.

The original bond-based peridynamic formulation employed a constant micromodulus within the horizon \cite{silling2000}, which remains used due to its simplicity.
However, later numerical studies and engineering implementations introduced smooth, compactly supported scalar kernels to improve quadrature accuracy, mitigate partial-volume effects, and enhance convergence properties \cite{seleson2016convergence}.
In contrast, the tensor kernel, also originating from the bond-based theory \cite{silling2000, silling2007peridynamic}, plays a fundamental theoretical role by enforcing isotropy and recovering classical elasticity in the vanishing-horizon limit.

\subsubsection{Scalar kernel}\label{Subsec:scalarkernel}
The scalar kernel $s(\boldsymbol{\xi})$ considered in this work is a radial, compactly supported micromodulus of the form
\begin{equation}\label{scalar_kernel}
	s(\boldsymbol{\xi}):=
	\begin{cases}
		c_1(\delta)\Bigl(1-\dfrac{|\boldsymbol{\xi}|}{\delta}\Bigr), & |\boldsymbol{\xi}|\le\delta,\\[0.6em]
		0,& |\boldsymbol{\xi}|>\delta,
	\end{cases}
\end{equation}
where $c_{1}(\delta)$ is a normalization constant depending on $\delta$. When
\begin{equation*}
	c_1(\delta)=\frac{20}{\pi\delta^{4}}
\end{equation*}
the kernel satisfies the two fundamental requirements for microelastic materials stated in \cite{silling2003deformation}: (i) the symmetry condition $s(\boldsymbol{\xi})=s(-\boldsymbol{\xi})$, which ensures compliance with Newton's third law;
(ii) the positivity of the energy associated with any sinusoidal deformation, expressed as
\begin{equation*}
	\int_{\mathbb{R}^2} \bigl(1-\cos(\kappa\cdot\boldsymbol{\xi})\bigr)s(\boldsymbol{\xi})\,d\boldsymbol{\xi}>0,\qquad \forall\,\kappa\neq 0.
\end{equation*}
Under this scalar kernel, the peridynamic equation of motion takes the form
\begin{equation*}
	\rho\,\ddot{\mathbf{u}}(\mathbf{x},t) =
	\int_{\mathcal{H}(\mathbf{x},\delta)}s(\boldsymbol{\xi})\bigl(\mathbf{u}(\mathbf{y},t)-\mathbf{u}(\mathbf{x},t)\bigr)\,d\mathbf{y} + \mathbf{b}(\mathbf{x},t).
\end{equation*}
The associated steady-state problem is given by
\begin{equation*}
	\mathcal{L}_{1,\delta} \mathbf{u}(\mathbf{x}) := -\int_{\mathcal{H}(\mathbf{x},\delta)}s(\boldsymbol{\xi})\bigl(\mathbf{u}(\mathbf{y})-\mathbf{u}(\mathbf{x})\bigr)\,d\mathbf{y}.
\end{equation*}
\subsubsection{Tensor kernel}\label{Subsec:tensorkernel}
In the realm of engineering applications, the tensor peridynamic model is more extensively employed. Within this framework, the pairwise interaction force density vector acting between two material points $\mathbf{x}$ and $\mathbf{y}$ is formulated as 
\begin{equation*} 
	\mathbf{f}(\mathbf{x},\mathbf{y},t) = c_{2}(\delta)\frac{\boldsymbol{\xi}\otimes\boldsymbol{\xi}}{|\boldsymbol{\xi}|^3}\bigl(\mathbf{u}(\mathbf{y},t)-\mathbf{u}(\mathbf{x},t)\bigr), \qquad \boldsymbol{\xi}=\mathbf{y}-\mathbf{x},
\end{equation*} 
where $\boldsymbol{\xi}\otimes\boldsymbol{\xi}$ denotes the second-order tensor defined by the dyadic product 
\begin{equation*}
	\boldsymbol{\xi}\otimes\boldsymbol{\xi} =
	\begin{pmatrix}
		\xi_1^2 & \xi_1\xi_2\\
		\xi_1\xi_2 & \xi_2^2
	\end{pmatrix},\qquad \|\boldsymbol{\xi}\|=\sqrt{\xi_1^2+\xi_2^2},
\end{equation*}
The coefficient $c_{2}(\delta)$ represents a material parameter. Adopting the standard calibration conventions found in the literature, $c_{2}(\delta)$ is given by 
\begin{equation*}
	c_{2}(\delta)=
	\begin{cases}
		\dfrac{72\kappa}{5\pi\delta^3}, & d=2,\\[0.4em]
		\dfrac{18\kappa}{\pi\delta^4}, & d=3,
	\end{cases}
\end{equation*}
where $\kappa$ is a material-dependent constant.
The tensor kernel is denoted as 
\begin{equation}\label{tensor_kernel}
	\mathbf{T}(\boldsymbol{\xi}):=c_{2}(\delta)\frac{\boldsymbol{\xi}\otimes\boldsymbol{\xi}}{\|\boldsymbol{\xi}\|^3}.
\end{equation}
Incorporating this tensor kernel, the governing equation of motion for linear elastic peridynamics is expressed as
\begin{equation*}
	\rho\,\ddot{\mathbf{u}}(\mathbf{x},t) =
	\int_{\mathcal{H}(\mathbf{x},\delta)}\mathbf{T}(\boldsymbol{\xi})\bigl(\mathbf{u}(\mathbf{y},t)-\mathbf{u}(\mathbf{x},t)\bigr)d\mathbf{y} + b(\mathbf{x},t),
\end{equation*}
and the corresponding steady-state operator is defined as
\begin{equation*}
	\mathcal{L}_{2,\delta} \mathbf{u}(\mathbf{x}) :=-\int_{\mathcal{H}(\mathbf{x},\delta)}\mathbf{T}(\boldsymbol{\xi})\bigl(\mathbf{u}(\mathbf{y})-\mathbf{u}(\mathbf{x})\bigr)d\mathbf{y}.\label{tensor}
\end{equation*}
The tensor kernel model effectively captures the stiffness coupling across different directions, making it a more generalized peridynamic description for two-dimensional linear elasticity problems. For the sake of notational distinction within this work, the steady-state operators associated with the scalar and tensor kernels are denoted by $\mathcal{L}_{1,\delta},\mathcal{L}_{2,\delta}$, respectively. 

\section{Numerical Implementation via IPA-AC Algorithm}\label{Sec:IPA-ACalgorithm}
\subsection{Spatial discretization and nodal representation}
In the framework of meshfree methods, the continuous body $\mathcal{B}$ is discretized into a finite set of material points, called nodes. The set of these nodes is denoted by $\mathcal{N}_{\mathcal{B}}$. For each node $\mathbf{x}_{i}\in\mathcal{N}_{\mathcal{B}}$, we define a representative cell $\tau_{i}\subset(\mathcal{B}\cap\hat{\mathcal{B}})$ of finite measure such that $\mathbf{x}_{i}$ is the centroid of $\tau_{i}$. The properties on $\tau_{i}$ are approximated by the discrete values assigned at $\mathbf{x}_{i}$.

As a meshfree framework, the IPA-AC method does not rely on a topological background mesh for element connectivity. However, each node is associated with a cell. While these cells can theoretically assume arbitrary shapes, standard numerical implementations-including the theoretical derivations presented here-employ a regular discretization using square cells (in 2D). Consequently, the reference configuration is discretized into a collection of these regular cells. The edge length of such a cell is denoted by $h$. For the sake of convention and without ambiguity, we refer to $h$ as the mesh size. It serves as a metric for the discretization resolution, where a smaller $h$ corresponds to a finer representation of the continuum body.

By partitioning the entire physical domain into these disjoint cells, the continuous nonlocal operator acting on $\mathbf{x}_{i}$ can be expressed as a sum of integrals over the individual cells within its interaction neighborhood, which are subsequently approximated using a single-point quadrature rule.
\begin{equation*}
	\begin{aligned}
		&\int_{\mathcal{H}(\mathbf{x}_{i},\delta)}\mathbf{C}(\boldsymbol{\xi})
		\left(\mathbf{u}(\mathbf{x}',t)-\mathbf{u}(\mathbf{x}_{i},t)\right) d\mathbf{x}'\\ 
		&=\sum_{j}\int_{\mathcal{H}(\mathbf{x}_{i},\delta)\cup\tau_{j}}\mathbf{C}(\boldsymbol{\xi})
		\left(\mathbf{u}(\mathbf{x}',t)-\mathbf{u}(\mathbf{x}_{i},t)\right) d\mathbf{x}'\\
		&\approx\sum_{j\in\mathcal{F}_{i}}\mathbf{C}(\boldsymbol{\xi})
		\left(\mathbf{u}(\mathbf{x}_{j},t)-\mathbf{u}(\mathbf{x}_{i},t)\right)\mathcal{A}_{j}^{i}.
	\end{aligned}
\end{equation*}
where $\mathcal{F}_{i}$ denotes the set of neighboring nodes within the horizon of $\mathbf{x}_{i}$, $\mathbf{x}_{j}$ denotes the centroid of cell $\tau_{j}$, and $\mathcal{A}_{j}^{i}$ represents the geometric measure of the intersection between the cell $\tau_{j}$ and the interaction domain $\mathcal{H}(\mathbf{x}_{i},\delta)$.
This nodal representation transforms the continuous governing equations into a discrete system of interactions, providing the basis for the subsequent IPA-AC operator formulation.
\subsection{The IPA-AC algorithm}\label{Subsec:IPA-AC}

The IPA-AC algorithm is an optimization strategy designed to enhance the accuracy of discrete nonlocal interaction operators within meshfree discretization frameworks. Unlike traditional collocation-based meshfree methods that approximate domain integrals using fixed nodal volumes, the IPA-AC algorithm introduces a sophisticated geometric treatment of the discrete interaction domain.

\begin{figure}
	\centering
	\subfigure[FA algorithm: Identification of effective cells (\textasteriskcentered) and excluded cells ($\circ$) based on whether the cell centroid lies within the neighborhood.]
	{\includegraphics[width=0.45\linewidth]{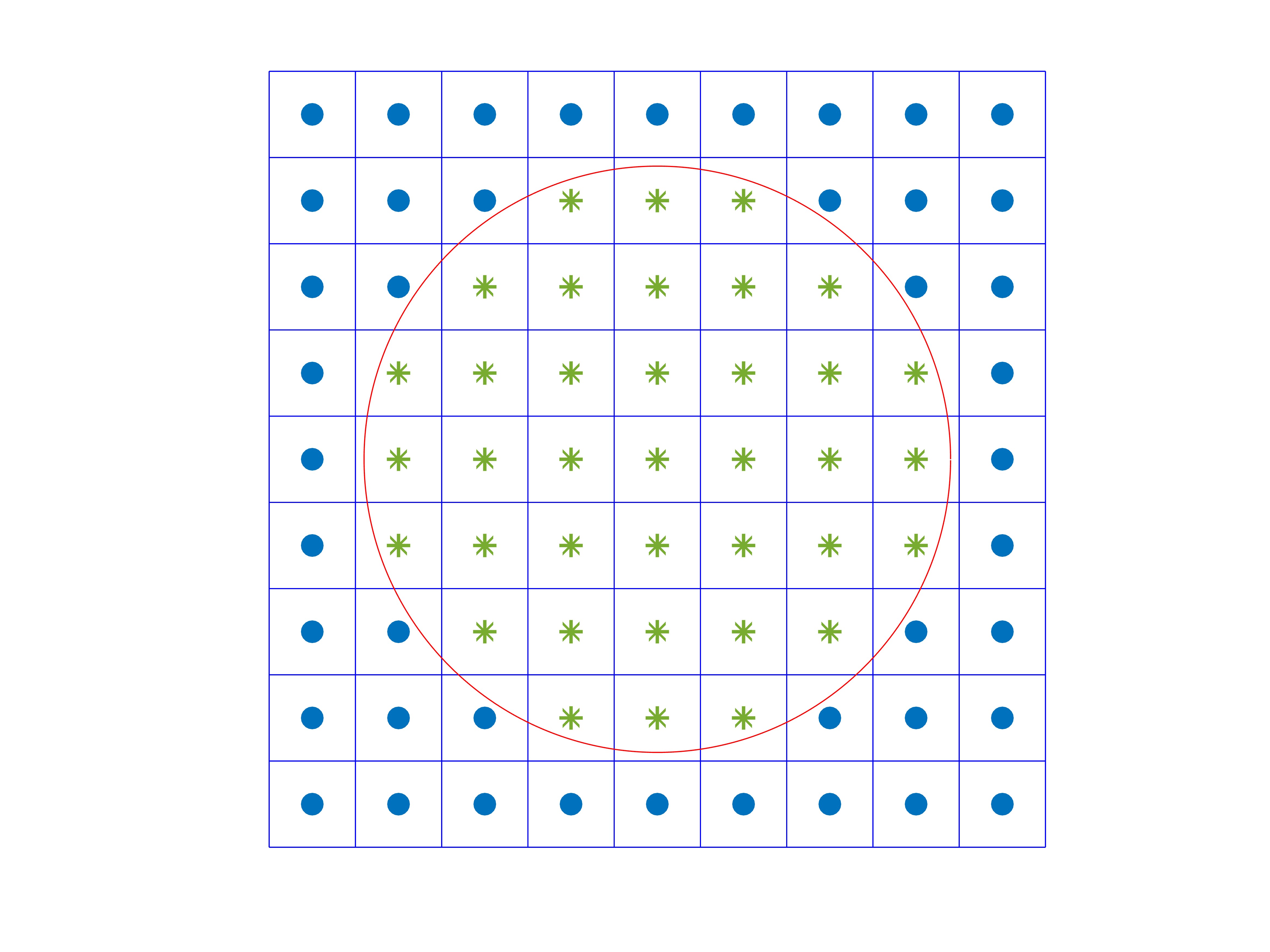}
		\put(-91,66){\tiny$\mathbf{x}_{i}$}	
		\label{fig:1a}}
	\hfill
	\subfigure[PA-AC algorithm: Categorization of material points into full cells ($\square$), standard partial cells ($\lozenge$), boundary partial cells ($\triangle$), and excluded cells ($\circ$).]
	{\includegraphics[width=0.45\linewidth]{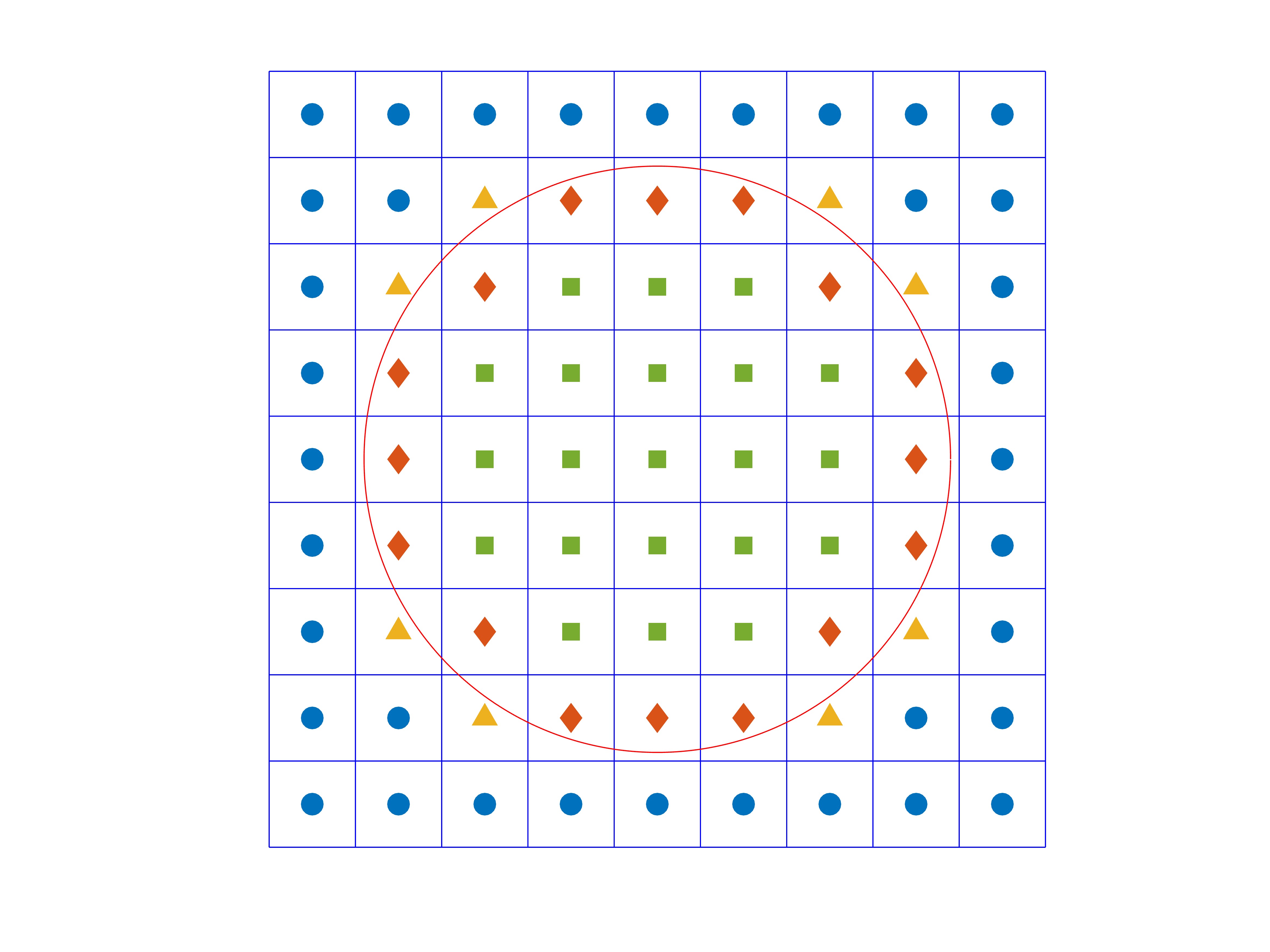}
		\put(-91,66){\tiny$\mathbf{x}_{i}$}
		\label{fig:1b} }
	\caption{Classification of material points within the interaction neighborhood of point $\mathbf{x}_{i}$. The interior of the circle denotes the interaction domain $\mathcal{H}$.}
	\label{class}
\end{figure}

In conventional meshfree discretizations, such as the FA algorithm \cite{silling2005meshfree}, the inclusion of material points is typically governed by a binary criterion: cells are considered effective only if their centroids lie within the neighborhood $\mathcal{H}$, while all others are excluded, as illustrated in \cref{fig:1a}. To overcome the geometric bias inherent in this approach, the IPA-AC algorithm systematically categorizes cells into four distinct types based on their geometric intersection with the neighborhood and the spatial relationship of their centroids relative to the horizon boundary as shown
in \cref{fig:1b}:

Full Cells (Square markers): Cells entirely contained within the horizon, where the effective area equals the full cell area.

Standard Partial Cells (Diamond markers): Cells with centroids inside the horizon but truncated by its boundary. While the FA algorithm and certain implementations like LAMMPS \cite{parks2008implementing} typically employ either the full cell area or simple scaling factors, the IPA-AC method evaluates their exact analytical intersection area.

Boundary Partial Cells (Triangle markers): Cells whose centroids lie outside the horizon yet partially intersect the interaction domain. These cells are frequently neglected in conventional discretizations, which introduces a systematic geometric bias and compromises convergence.

Excluded Cells (Circle markers): Cells situated entirely outside the interaction range with no geometric contribution.

Based on this classification, the IPA-AC algorithm is fundamentally characterized by two core features:

Exact Evaluation of Effective Areas (PA-AC component): The Partial Area-Analytical Calculation (PA-AC) precisely determines the effective area of each cell truncated by the horizon boundary. Crucially, it explicitly incorporates Boundary Partial Cells, ensuring that the geometric measure of the interaction domain is preserved with high fidelity.

Building upon the PA-AC framework, the IPA-AC algorithm minimizes integration errors by shifting the quadrature point from the original cell center to the centroid of the newly computed effective area, as shown in \cref{fig:IPA-AC}. This adjustment ensures the discrete force contributions are strictly consistent with the actual geometry of the nonlocal interaction domain, thereby mitigating surface effects.

\begin{figure}
	\centering
	\includegraphics[width=0.45\linewidth]{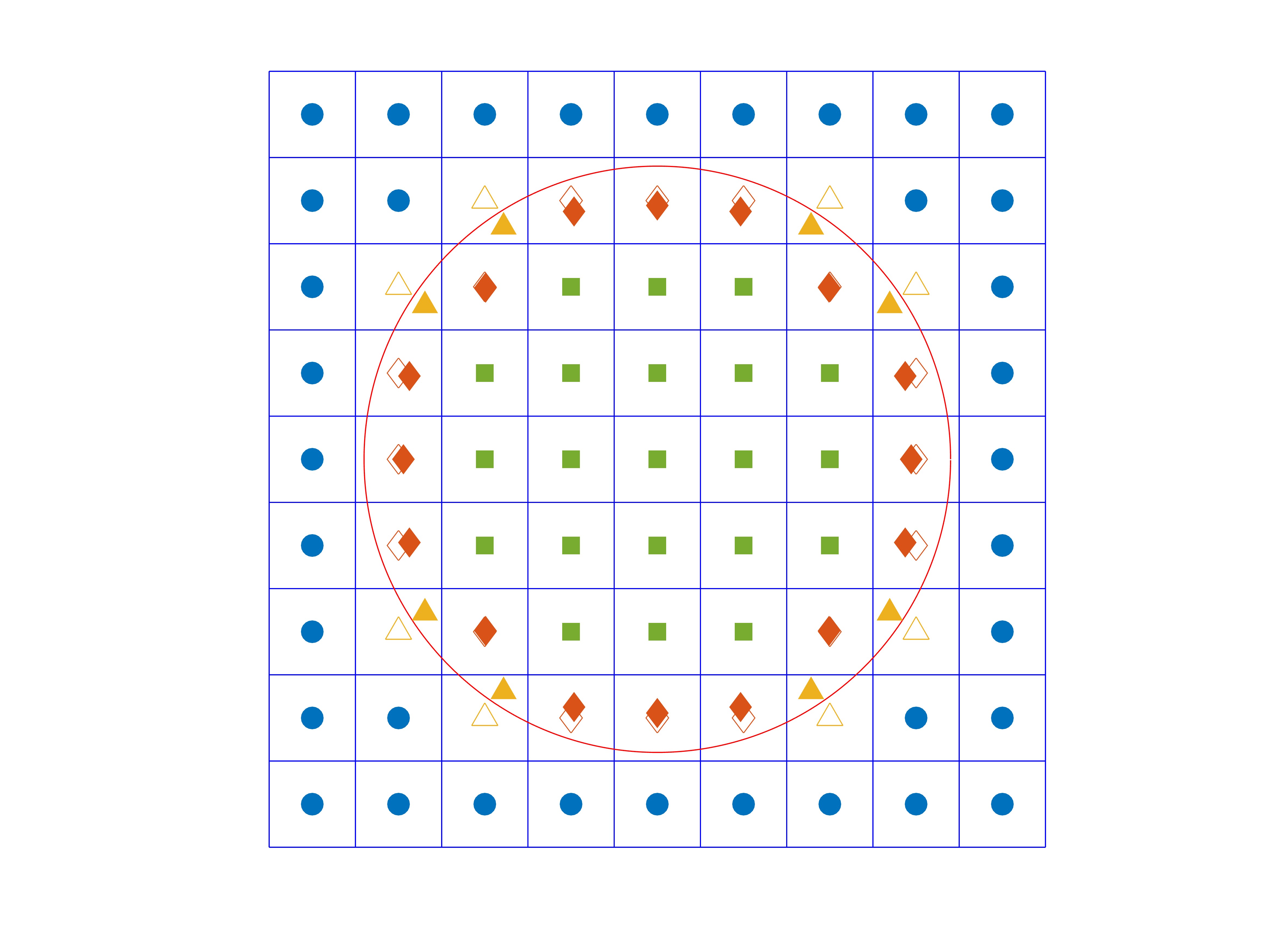}
	\put(-91,66){\tiny$\mathbf{x}_{i}$}
	\caption{IPA-AC algorithm: Shifting quadrature points of partial cells to the centroids of the intersected regions.}
	\label{fig:IPA-AC}
\end{figure}

Compared to standard meshfree schemes, the primary advantage of IPA-AC lies in its ability to maintain the consistency of the discrete operator near the horizon boundary. By explicitly accounting for both the measure and the first moment of the intersected regions, it provides a superior balance between computational efficiency and numerical convergence. As will be shown in the subsequent analysis, this precise geometric treatment of the centroid is the key factor in reducing the local truncation error term derived in \cref{Sec:Convergence}.

\section{Convergence Analysis}\label{Sec:Convergence}
Building upon the geometric corrections introduced by the IPA-AC framework, we now establish a rigorous theoretical characterization of the method’s convergence behavior. Although prior studies \cite{seleson2014improved,seleson2016convergence} have discussed the algorithmic convergence with respect to the mesh size $h$ under a fixed horizon, a comprehensive investigation of the convergence with respect to the horizon $\delta$ and the AC property remains to be supplemented.

To bridge this gap, this section systematically analyzes the error estimates for both scalar and tensor nonlocal operators. In this section, we explicitly derive the convergence orders of the discrete operator with respect to both the horizon size $\delta$ and the grid spacing $h$. These specific error estimates allow us to provide a comprehensive clarification of the three representative convergence regimes discussed in Section 1:(i) Standard refinement; (ii) $\delta$ refinement;(iii) AC property.

\subsection{Error decomposition Framework}\label{sec:Error_decomposition}
To facilitate a unified analysis for both scalar and tensor kernels, we first establish a general error decomposition framework based on the properties of the discrete linear operator. Let $\mathbf{u}_{\delta}$ denote the exact solution of the nonlocal problem and $\mathbf{u}_{\delta}^{h}$ represent the numerical solution obtained via the IPA-AC method. The discrete linear system governing the steady-state deformation is given by
\begin{equation}\label{eq:steady-state}
	\mathcal{L}_{\delta}^{h} \mathbf{u}_{\delta}^{h}(\mathbf{x}_{i}) := -\sum_{j\in\mathcal{F}_{i}}\mathbf{C}(\boldsymbol{\xi})\left(\mathbf{u}_{\delta}^{h}(\mathbf{x}_{j})-\mathbf{u}_{\delta}^{h}(\mathbf{x}_{i})\right)\mathcal{A}_{j}^{i} = \mathbf{b}(\mathbf{x}_{i}), \quad \forall \mathbf{x}_{i} \in {\mathcal{B}},
\end{equation}
 where $\mathcal{L}_{\delta}^{h}$ denotes the discrete operator (encompassing both scalar and tensor formulations) and $\mathbf{b}$ is the body force density. In this study, we employ the infinity norm $\|\cdot\|_{\infty}$ to quantify the discrepancy between the exact and numerical solutions. By invoking the linearity of the operator, the global error bound is derived as follows  
\begin{align}
	\label{eq:errordecomposition} 
	\|\mathbf{u}_{\delta}-\mathbf{u}_{\delta}^{h}\|_{\infty}
	&=\|\mathbf{u}_{\delta}-(\mathcal{L}_{\delta}^{h})^{-1} \mathbf{b}\|_{\infty}
	=\|(\mathcal{L}_{\delta}^{h})^{-1}(\mathcal{L}_{\delta}^{h}\mathbf{u}_{\delta}-\mathbf{b})\|_{\infty} \nonumber\\
	&=\|(\mathcal{L}_{\delta}^{h})^{-1}(\mathcal{L}_{\delta}^{h}\mathbf{u}_{\delta}-\mathcal{L}_{\delta}\mathbf{u}_{\delta})\|_{\infty} \nonumber\\
	&\leq\|(\mathcal{L}_{\delta}^{h})^{-1}\|_{\infty}\|(\mathcal{L}_{\delta}\mathbf{u}_{\delta}-\mathcal{L}_{\delta}^{h}\mathbf{u}_{\delta})\|_{\infty}
\end{align}
This inequality, rooted in the Lax equivalence framework, effectively decouples the convergence analysis into two independent tasks:

Stability Analysis: Proving the uniform boundedness of the inverse operator norm $\|(\mathcal{L}_{\delta}^{h})^{-1}\|_{\infty}$.

Consistency Analysis: Estimating the local truncation error defined as $\|(\mathcal{L}_{\delta}\mathbf{u}_{\delta}-\mathcal{L}_{\delta}^{h}\mathbf{u}_{\delta})\|_{\infty}$.

With the general sufficient conditions for convergence established in \cref{eq:errordecomposition}, we now proceed to verify these conditions for specific kernel formulations. The analysis is structured hierarchically: we first investigate the scalar kernel in \cref{sec:Convergence scalar kernel}. Subsequently, in \cref{sec:Convergence tensor kernel}, we extend this theoretical rigor to the tensor kernel to address the additional complexities introduced by directional dependencies.

To simplify the notation in the following proof, we introduce the shorthand: 
\begin{equation*}
	\mathbf{u}=u_{a},\quad \text{for } a \in \{1,2\}
\end{equation*}
and
\begin{equation*} 
	\partial_{i} = \frac{\partial}{\partial y_{i}}, \quad \partial_{ij}^{2} = \frac{\partial^{2}}{\partial y_{i} \partial y_{j}}, \quad \text{for } i,j \in \{1,2\}.
\end{equation*} 
Additionally, the Einstein summation convention is adopted, implying summation over repeated indices.

\subsection{Convergence analysis of scalar kernel}\label{sec:Convergence scalar kernel}
In this subsection, we specialize the error decomposition framework established in \cref{sec:Error_decomposition} to the scalar kernel model. Specifically, we aim to demonstrate that the discrete scalar kernel operator $\mathcal{L}_{1,\delta}^{h}$ satisfies the two sufficient conditions for convergence: (1) stability is ensured by the structural properties of the resulting stiffness matrix, and (2) consistency is achieved via the IPA-AC geometric corrections.

\subsubsection{Stability analysis of scalar kernel}\label{Subsubsec:Stability analysis of scalar kernel}
The \cref{eq:steady-state} represents the nodal force balance for an arbitrary node $\mathbf{x}_{i}$. The assembly of these equations for all nodes $\mathbf{x}_{i} \in {\mathcal{B}}$ constitutes the global linear system
\begin{equation*} 
	\mathbf{K} \mathbf{U} = \mathbf{F}, 
\end{equation*} 
where $\mathbf{U}$ is the global solution vector and $\mathbf{K}$ is the stiffness matrix associated with the discrete operator $\mathcal{L}_{\delta}^{h}$. Consequently, the stability of the numerical scheme relies on establishing two properties: the non-singularity of $\mathbf{K}$ and the uniform boundedness of its inverse with respect to the mesh size $h$.

\begin{lemma}(Invertibility)\label{Invertibility} Provided that the horizon $\delta$ is sufficiently resolved by the mesh to ensure graph connectivity, the discrete stiffness matrix $\mathbf{K}$ is a non-singular M-matrix. Consequently, the inverse operator $(\mathcal{L}_{\delta}^{h})^{-1}$ exists and is non-negative.
\end{lemma}

\begin{proof} We verify that $\mathbf{K}$  satisfies the structural properties of an M-matrix based on the non-negativity of the scalar kernel $s(\boldsymbol{\xi})$:
	
	Off-diagonal Non-positivity: For $i\neq j$, the entries are $\mathbf{K}_{ij}=-\mathbf{S}(\mathbf{x}_{j}-\mathbf{x}_{i})\mathcal{A}_{j}^{i}\leq0$.
	
	Weak Diagonal Dominance for Internal Nodes: The diagonal entries satisfy $|\mathbf{K}_{ii}|=\sum_{j\neq i}|\mathbf{K}_{ij}|>0$.
	
	Strict Diagonal Dominance at Boundary: For nodes located within a distance $\delta$ from the boundary, the interaction domain intersects with the fictitious layer where Dirichlet constraints are imposed. Since the corresponding off-diagonal entries are moved to the right-hand side, we have strict inequality: $|\mathbf{K}_{ii}|>\sum_{j\neq i}|\mathbf{K}_{ij}|$.
	
	Since the physical domain is connected and the horizon overlaps sufficiently ($\delta>h$), the matrix $\mathbf{K}$ is irreducible. Standard matrix theory states that an irreducible, diagonally dominant matrix with non-positive off-diagonals and at least one strictly dominant row is non-singular, and its inverse consists of non-negative entries ($\mathbf{K}^{-1}>0$). 
\end{proof}

Furthermore, assuming the physical domain and the mesh cells are regular, the resulting stiffness matrix $\mathbf{K}$ is symmetric. Consequently, under these conditions, $\mathbf{K}$ is a Stieltjes matrix (a symmetric M-matrix), which possesses real and positive eigenvalues.

\begin{lemma}\label{Uniform Boundedness of Scalar kernel}
	(Uniform Boundedness of Scalar kernel) The inverse of the discrete operator is uniformly bounded independent of the mesh size $h$ and horizon size $\delta$. That is, there exists a constant $C_{s1}$ independent of the mesh size $h$ such that: 
	\begin{equation} 
		\lVert(\mathcal{L}_{\delta}^{h})^{-1}\rVert_{\infty} \leq C_{s1}. 
	\end{equation}
\end{lemma}

\begin{proof}
	Let $\mathbf{v}_{h}$ ​	
	be a vector satisfying the equation 
	\begin{equation*}
		\mathcal{L}_{\delta}^{h}\mathbf{v}_{h} = \boldsymbol{1}.
	\end{equation*}
	By virtue of the non-negativity of the inverse operator established in \cref{Invertibility}, the infinity norm of $(\mathcal{L}_{\delta}^{h})^{-1}$ is determined by its action on the unit vector: 
	\begin{equation*} 
		\lVert(\mathcal{L}_{\delta}^{h})^{-1}\rVert_{\infty}
		=\lVert(\mathcal{L}_{\delta}^{h})^{-1}\boldsymbol{1}\rVert_{\infty} 
		=\lVert\mathbf{v}_{h}\rVert_{\infty}. 
	\end{equation*} 
	Consequently, proving the stability of the scheme is equivalent to demonstrating that $\lVert\mathbf{v}_{h}\rVert_{\infty}$ is bounded. The problem thus reduces to constructing a function $\mathbf{v}_{h}$ that satisfies:
	$\mathcal{L}_{\delta}^{h}\mathbf{v}_{h} = \boldsymbol{1}$,
	$\lVert\mathbf{v}_{h}\rVert_{\infty}\leq C$
	for some constant $C$ independent of $h$.
	To prove the uniform boundedness of the inverse operator, we construct an explicit auxiliary function. Let the scalar function $v(\mathbf{x}),~a=\{1,2\}$ be a quadratic polynomial.
	$$v(\mathbf{x})=\frac{x_{1}(1-x_{1})}{2} 
	+\frac{x_{2}(1-x_{2})}{2}.$$
	We evaluate the action of the discrete operator $\mathcal{L}_{\delta}^{h}$ on $v(\mathbf{x}_{i})$.
	Exploiting the geometric symmetry of the uniform grid, we rewrite the summation over the full neighborhood as a summation over the 'half-neighborhood'. Let $\mathcal{F}_{i}^{+}$​	denote the set of positive integer offsets, defined as
	$$\mathcal{F}_{i}^{+} :=
	\{m\in Z^{+}|\mathbf{x}_{i+m}\in\mathcal{F}_{i}\}.$$
	For every $m\in\mathcal{F}_{i}^{+}$, there exists a corresponding neighbor $\mathbf{x}_{i-m}$ symmetrically located with respect to $\mathbf{x}_{i}$ as shown in \cref{fig:symmetric}. Consequently, the operator can be reformulated by pairing these symmetric interactions
	
	\begin{figure}
		\centering
		\begin{tikzpicture}
			\node[anchor=south west] (image) at (0,0) {\includegraphics[width=0.6\textwidth]{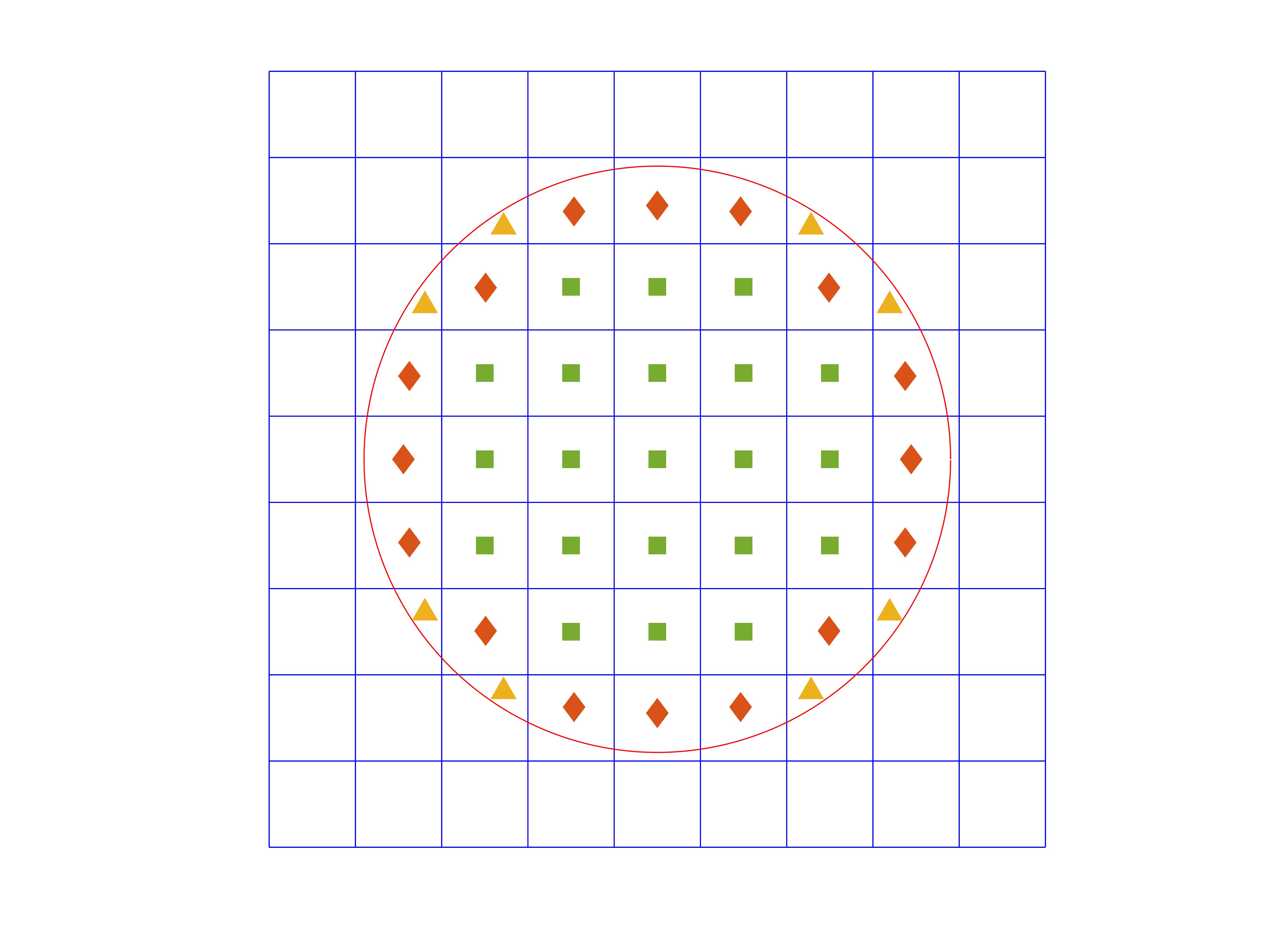}};
			\draw[->, line width=1pt] (4.65,3.5) -- (3.48,4.66);  
			\draw[->, line width=1pt] (4.65,3.5) -- (5.81,2.36);
			\node[above] at (3.9,3.65){\textbf{$\mathbf{a}_{m}$}};
			\node[above] at (5,2.4){\textbf{$-\mathbf{a}_{m}$}};
			\node[above] at (4.86,3.3){\textbf{$\mathbf{x}_{i}$}};
			\node[above] at (6,1.86){\textbf{$\mathbf{x}_{i-m}$}};
			\node[above] at (3.72,4.56){\textbf{$\mathbf{x}_{i-m}$}};
		\end{tikzpicture}
		\caption{Geometric symmetry of the neighborhood of node $\mathbf{x}_{i}$. Each point $\mathbf{x}_{i+m}$ admits a symmetric counterpart $\mathbf{x}_{i-m}$ with an identical cell area. The two vectors formed by each symmetric pair and the source point $\mathbf{x}_{i}$ have equal magnitudes and opposite directions.}
		\label{fig:symmetric}	
	\end{figure}
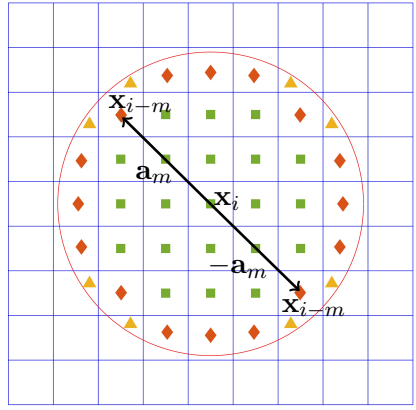
	$$\begin{aligned}
		\mathcal{L}_{\delta}^{h}v(\mathbf{x}_{i})
		&=-c_{1}(\delta)\sum_{p\in\mathcal{F}_{i}}\left(v(\mathbf{x}_{p})-v(\mathbf{x}_{i})\right)\left(1-\frac{|\mathbf{x}_{p}-\mathbf{x}_{i}|}{\delta}\right)\mathcal{A}_{p}^{i}\\
		&=c_{1}(\delta)\sum_{m\in\mathcal{F}_{i}^{+}}\left(-v(\mathbf{x}_{i}-\mathbf{a}_{m})+2v(\mathbf{x}_{i})-v(\mathbf{x}_{i}+\mathbf{a}_{m})\right)\left(1-\frac{|\mathbf{a}_{m}|}{\delta}\right)\mathcal{A}_{i+m}^{i},
	\end{aligned}$$
	where $\mathbf{a}_{m}=\mathbf{x}_{i+m}-\mathbf{x}_{i}$ denotes the relative position vector of the neighbor. For the chosen quadratic polynomial, the second-order central difference yields an exact analytical expression. Substituting $v(\mathbf{x})$ into the bracketed term, we get
	$$\begin{aligned}
		&-v(\mathbf{x}_{i}-\mathbf{a}_{m})+2v(\mathbf{x}_{i})-v(\mathbf{x}_{i}+\mathbf{a}_{m})\\
		&=-\frac{1}{2}\left[(x_{i,1}-a_{m,1})(1-x_{i,1}+a_{m,1})+(x_{i,2}-a_{m,2})(1-x_{i,2}+a_{m,2})\right]\\
		&+x_{i,1}(1-x_{i,1})+x_{i,2}(1-x_{i,2})\\
		&-\frac{1}{2}\left[(x_{i,1}+a_{m,1})(1-x_{i,1}-a_{m,1})+(x_{i,2}+a_{m,2})(1-x_{i,2}-a_{m,2})\right]\\
		&=a_{m,1}^{2}+a_{m,2}^{2}\\
		&=|\mathbf{a}_{m}|^{2}.
	\end{aligned}$$
	It is evident that the second-order difference of the quadratic function depends only on the square of the relative distance $|\mathbf{a}_{m}|$ and is independent of the center $\mathbf{x}_{i}$. Substituting this result back into the operator expression yields a constant $C_{\delta,h}$,
	$$\begin{aligned}
		\mathcal{L}_{\delta}^{h}v(\mathbf{x}_{i})=c_{1}(\delta)\sum_{m\in\mathcal{F}_{i}^{+}}|\mathbf{a}_{m}|^{2}\left(1-\frac{|\mathbf{a}_{m}|}{\delta}\right)\mathcal{A}_{i+m}^{i}=C_{\delta,h}.
	\end{aligned}$$
	To satisfy the condition $\mathcal{L}_{\delta}^{h}\mathbf{v}_{h} = \boldsymbol{1}$, we define the normalized vector field
	\begin{equation}
		\mathbf{v}_{h}=\frac{1}{C_{\delta,h}}
		\left(\begin{array}{c}
		v(\mathbf{x})\\v(\mathbf{x})
		\end{array}\right).\label{Vh}
	\end{equation}
	Given that $v(\mathbf{x})$ is bounded on the unit domain with a maximum value of $\frac{1}{4}$, the infinity norm is given by $\lVert\mathbf{v}_{h}\rVert_{\infty} = \frac{1}{4C_{\delta,h}}$.
	
	We now establish that $C_{\delta,h}$ is a positive constant strictly bounded away from zero, independent of $h$. The sum can be estimated as the product of the effective area $\pi\delta^{2}$ and the integrand evaluated at some intermediate point $\alpha\delta$ where $\alpha\in(0,1)$
	$$\begin{aligned}
		C_{\delta,h}&=c_{1}(\delta)\sum_{m}|\mathbf{a}_{m}|^{2}\left(1-\frac{|\mathbf{a}_{m}|}{\delta}\right)\mathcal{A}_{i+m}^{i}\\
		&=c_{1}(\delta)\rho(\alpha)\pi\delta^{2},
	\end{aligned}$$
	where $\rho(\alpha)=(\alpha\delta)^{2}\left(1-\frac{\alpha\delta}{\delta}\right)=\alpha^{2}(1-\alpha)\delta^{2}$. Let $\beta=\alpha^{2}(1-\alpha)$, $\rho(\alpha)=\rho(\beta)=\beta\delta^{2}$. Using the scaling coefficient for the 2D kernel, $c_{1}(\delta)=\frac{20}{\pi\delta^{4}}$, we obtain
	$$C_{\delta,h}=c_{1}(\delta)\beta\pi\delta^{4}=\frac{20}{\pi\delta^{4}}\beta\pi\delta^{4}=20\beta.$$
	Since the kernel is non-negative and supported on a set of non-zero measure, $\beta$ represents a strictly positive value bounded away from zero (specifically, $\beta\in(0,4/27)$). Consequently, $C_{\delta,h}$ is a constant independent of the discretization parameters, ensuring that $\lVert\mathbf{V}_{h}\rVert_{\infty}$ remains uniformly bounded.
\end{proof}

\subsubsection{Consistency analysis of scalar kernel}\label{Subsubsec:Consistency analysis of scalar kernel}
We first estimate the local truncation error for the scalar operator. The key innovation of the IPA-AC algorithm lies in its precise handling of the integration weights and quadrature points, which directly impacts the accuracy of the discrete approximation.

\begin{lemma}\label{Consistency of Scalar Operator}
(Consistency of Scalar Operator) Assume that the exact solution $\mathbf{u}\in C^{2}(\mathcal{B})$ and the scalar kernel $s(\boldsymbol{\xi})$ defined as
\cref{Subsec:scalarkernel}. Then, the local truncation error of the IPA-AC discrete operator is uniformly bounded by: 
\begin{equation*} 		\|\mathcal{L}_{1,\delta}\mathbf{u} - \mathcal{L}_{1,\delta}^{h}\mathbf{u}\|_{\infty} 
\leq C_{s2} \frac{h^{2}}{\delta^{2}},
\end{equation*} 
where $C_{s2}$ is a generic constant independent of the mesh size $h$ and the horizon $\delta$.
\end{lemma}
\begin{proof}
Let $\mathbf{y}^{cp}$ denote the centroid of the cell $\omega_{p}$.
\begin{equation*}
\begin{aligned}
	&\|\mathcal{L}_{1,\delta}\mathbf{u} - \mathcal{L}_{1,\delta}^{h}\mathbf{u}\|_{\infty}\\
	&\leq\|\sum_{p}\int_{\Omega_{p}}s(\boldsymbol{\xi})\left(u_{a}(\mathbf{y})-u_{a}(\mathbf{x}_{i})\right) d\mathbf{y}
	-\sum_{p\in\mathcal{F}_{i}}s(\boldsymbol{\xi})\left(u_{a}(\mathbf{y}^{cp})-u_{a}(\mathbf{x}_{i})\right)\mathcal{A}_{p}^{i}\|_{\infty}\\
	&\leq\underbrace{\|\sum_{p}\int_{\Omega_{p}}s(\mathbf{y}-\mathbf{x})\left(u_{a}(\mathbf{y})-u_{a}(\mathbf{x}_{i})\right) d\mathbf{y}
	-\sum_{p}\int_{\Omega_{p}}s(\mathbf{y}^{cp}-\mathbf{x})\left(u_{a}(\mathbf{y})-u_{a}(\mathbf{x}_{i})\right) d\mathbf{y}\|_{\infty}}_{\text{Part A}}\\
	&+\underbrace{\|\sum_{p}\int_{\Omega_{p}}s(\mathbf{y}^{cp}-\mathbf{x})\left(u_{a}(\mathbf{y})-u_{a}(\mathbf{x}_{i})\right) d\mathbf{y}
	-\sum_{p\in\mathcal{F}_{i}}s(\mathbf{y}^{cp}-\mathbf{x})\left(\mathbf{u}(\mathbf{y}^{cp})-u_{a}(\mathbf{x}_{i})\right)\mathcal{A}_{p}^{i}}_{\text{Part B}}
\end{aligned}
\end{equation*}
For Part A, we expand both the kernel function and the displacement field $\mathbf{u}$ with respect to $\mathbf{y}$ about $\mathbf{y}^{cp}$.
\begin{equation*}
\begin{aligned}
	A &= \|\sum_{p}\int_{\Omega_{p}}[s(\mathbf{x},\mathbf{y})-s(\mathbf{x},\mathbf{y}^{cp})][\left(u_{a}(\mathbf{y})-u_{a}(\mathbf{x}_{i})\right)] d\mathbf{y}\|_{\infty}\\
	&=\|\sum_{p}\int_{\Omega_{p}}[\partial_{i}s(\mathbf{x},\mathbf{y}^{cp})\Delta y_{i}^{cp}+\frac{1}{2}\partial_{i,j}^{2}s(\mathbf{x},\mathbf{y}^{cp})\Delta y_{i}^{cp}\Delta y_{j}^{cp}+\mathcal{O}(h^{3})]\\
	&[u_{a}(\mathbf{y}^{cp})-u_{a}(\mathbf{x}_{i})+\partial_{i}u_{a}(\mathbf{y}^{cp})\Delta y_{i}^{cp}+\frac{1}{2}\partial_{i,j}^{2}u_{a}(\mathbf{y}^{cp})\Delta y_{i}^{cp}\Delta y_{j}^{cp}+\mathcal{O}(h^{3})]d\mathbf{y}\|_{\infty}	
\end{aligned}
\end{equation*}
where $\Delta y_{\alpha}^{cp}=y_{\alpha}-y_{\alpha}^{cp},~\alpha=\{i,j\}$. By virtue of $\mathbf{y}^{cp}$	
being the geometric centroid of $\Omega_{p}$, the first moment of the domain vanishes. Consequently, the linear term in the expansion is eliminated upon integration, leaving the integrand dominated by second- and higher-order terms:
\begin{equation*}
\begin{aligned}
	A&=\|\sum_{p}\int_{\Omega_{p}}[\partial_{i}s(\mathbf{x},\mathbf{y}^{cp})\partial_{j}u_{a}(\mathbf{y}^{cp})\Delta y_{i}^{cp}\Delta y_{j}^{cp}\\
	&+\frac{1}{2}[u_{a}(\mathbf{y}^{cp})-u_{a}(\mathbf{x}_{i})][\partial_{i,j}^{2}s(\mathbf{x},\mathbf{y}^{cp})\Delta y_{i}^{cp}\Delta y_{j}^{cp}]+\mathcal{O}(h^{3})]d\mathbf{y}\|_{\infty}\\
	&=\|\sum_{p}\left\{\partial_{i}s(\mathbf{x},\mathbf{y}^{cp})\partial_{j}u_{a}(\mathbf{y}^{cp})+\frac{1}{2}[u_{a}(\mathbf{y}^{cp})-u_{a}(\mathbf{x}_{i})]\partial_{i,j}^{2}s(\mathbf{x},\mathbf{y}^{cp})\right\}\\
	&\int_{\Omega_{p}}\Delta y_{i}^{cp}\Delta y_{j}^{cp}d\mathbf{y}+\mathcal{O}(h^{5})\|_{\infty}.
\end{aligned}
\end{equation*}
With the notation $\Delta x_{k}=y_{k}^{cp}-x_{k}$,
we perform a Taylor expansion of $u_{a}(\mathbf{y}^{cp})$ around $\mathbf{x}_{i}$
\begin{equation*}
\begin{aligned}
	A&=\|\sum_{p}\{\partial_{i}s(\mathbf{x},\mathbf{y}^{cp})[\partial_{j}u_{a}(\mathbf{x})+\partial_{j,k}^{2}u_{a}(\mathbf{x})\Delta x_{k}+\mathcal{O}(\delta^{2})]\\
	&+\frac{1}{2}[\partial_{k}u_{a}(\mathbf{x})\Delta x_{k}+\frac{1}{2}\partial_{k,l}^{2}u_{a}(\mathbf{x})\Delta x_{k}\Delta x_{l}+\mathcal{O}(\delta^{3})][\partial_{i,j}^{2}s(\mathbf{x},\mathbf{y}^{cp})]\}\\
	&\int_{\Omega_{p}}\Delta y_{i}^{cp}\Delta y_{j}^{cp}d\mathbf{y}+\mathcal{O}(h^{5})\|_{\infty}
\end{aligned}
\end{equation*}
Since the kernel function is radial, its first and second derivatives satisfy the following properties, respectively
\begin{equation*}
	\partial_{i}s(\mathbf{x},\mathbf{y})=-\partial_{i}s(-\mathbf{x},-\mathbf{y}),\quad \partial_{i,j}^{2}s(\mathbf{x},\mathbf{y})=-\partial_{i,j}^{2}s(\mathbf{x},\mathbf{y}).
\end{equation*}
As illustrated in the \cref{fig:symmetric}, the neighborhood cells appear in pairs and exhibit central symmetry with respect to the source point $\mathbf{x}_{i}$. Consequently, the term $\partial_{i}s(\mathbf{x},\mathbf{y}^{cp})$ cancels out upon summation. Furthermore, due to this pairwise symmetry, the relative position vectors $\Delta x_{k}$ also appear in pairs with equal magnitude but opposite directions, ensuring that their contribution to the sum likewise vanishes. Consequently, the leading-order term of Part A reduces to
\begin{equation*}
\sum_{p}[\partial_{i}s(\mathbf{x},\mathbf{y}^{cp})\partial_{j,k}^{2}u_{a}(\mathbf{x})\Delta x_{k}+\frac{1}{4}\partial_{i,j}^{2}s(\mathbf{x},\mathbf{y}^{cp})\partial_{k,l}^{2}u_{a}(\mathbf{x})\Delta x_{k}\Delta x_{l}]\int_{\Omega_{p}}\Delta y_{i}^{cp}\Delta y_{j}^{cp}d\mathbf{y}\label{different}
\end{equation*}
Furthermore, since the scalar kernel is linear with respect to $\mathbf{y}$, its second derivative vanishes. Thus, we have:
\begin{equation*}
	A\sim\mathcal{O}\left(\frac{\delta^{2}}{h^{2}}[\delta^{-5}\cdot\delta+0]h^{4}\right)
	\sim\mathcal{O}\left(\frac{\delta^{2}}{h^{2}}\right).
\end{equation*}
Turning to Part B, we expand the displacement field $\mathbf{u}(\mathbf{y})$ in a Taylor series about the centroid $\mathbf{y}^{cp}$. the first moment of the relative position vector $\Delta y_{j}^{cp}$ evaluates to zero, leading to
\begin{equation*}
\begin{aligned}
	B&=\|\sum_{p}\int_{\Omega_{p}}s(\mathbf{y}^{cp}-\mathbf{x})\left[u_{a}(\mathbf{y})-u_{a}(\mathbf{y}^{cp})\right] d\mathbf{y}\|_{\infty}\\
	&=\|\sum_{p}\int_{\Omega_{p}}s(\mathbf{y}^{cp}-\mathbf{x})[\partial_{i}u_{a}(\mathbf{y}^{cp})\Delta y_{i}^{cp}+\frac{1}{2}\partial_{i,j}^{2}u_{a}(\mathbf{y}^{cp})\Delta y_{i}^{cp}\Delta y_{j}^{cp}+\mathcal{O}(h^{3})]d\mathbf{y}\|_{\infty}\\
	&=\|\sum_{p}s(\mathbf{y}^{cp}-\mathbf{x})\partial_{i}u_{a}(\mathbf{y}^{cp})\int_{\Omega_{p}}\Delta y_{i}^{cp}d\mathbf{y}\\
	&+\frac{1}{2}\sum_{p}s(\mathbf{y}^{cp}-\mathbf{x})\partial_{i,j}^{2}u_{a}(\mathbf{y}^{cp})\int_{\Omega_{p}}\Delta y_{i}^{cp}\Delta y_{j}^{cp}d\mathbf{y}+\sum_{p}s(\mathbf{y}^{cp}-\mathbf{x})\mathcal{O}(h^{5})\|_{\infty}\\
	&=\frac{1}{2}\|\sum_{p}s(\mathbf{y}^{cp}-\mathbf{x})\partial_{i,j}^{2}u_{a}(\mathbf{y}^{cp})\int_{\Omega_{p}}\Delta y_{i}^{cp}\Delta y_{j}^{cp}d\mathbf{y}+\sum_{p}s(\mathbf{y}^{cp}-\mathbf{x})\mathcal{O}(h^{5})\|_{\infty}\\
	&\sim\mathcal{O}\left(\frac{\delta^{2}}{h^{2}}\delta^{-4}h^{4}\right)\sim\mathcal{O}\left(\frac{\delta^{2}}{h^{2}}\right).
\end{aligned}
\end{equation*}
Combining the results from Parts A and B, we arrive at
\begin{equation*}
		\|\mathcal{L}_{1,\delta}\mathbf{u} - \mathcal{L}_{1,\delta}^{h}\mathbf{u}\|_{\infty} \sim\mathcal{O}\left(\frac{\delta^{2}}{h^{2}}\right).
\end{equation*}
\end{proof}

\subsubsection{Convergence analysis of scalar kernel}
By synthesizing the uniform stability established in \cref{Subsubsec:Stability analysis of scalar kernel} and the consistency results from \cref{Subsubsec:Consistency analysis of scalar kernel}, we arrive at the central theoretical result of this study. According to the Lax Equivalence Theorem, the convergence of the discrete solution $\mathbf{u}_{\delta}^{h}$​	to the peridynamic solution $\mathbf{u}_{\delta}$ is directly governed by the order of the local truncation error and the boundedness of the inverse operator.
\begin{theorem}\label{Convergence of Scalar Kernel}
(Convergence of Scalar Kernel) Let $\mathbf{u}_{\delta}$ be the solution to the continuous peridynamic equation and $\mathbf{u}_{\delta}^{h}$ be the corresponding discrete solution. For the scalar-kernel-based discretization proposed in this work, we have $$\|\mathbf{u}_{\delta}-\mathbf{u}_{\delta}^{h}\|_{\infty}\sim\mathcal{O}\left(\frac{h^{2}}{\delta^{2}}\right).$$
\end{theorem}
\begin{proof}
	\begin{align}
		\|\mathbf{u}_{\delta}-\mathbf{u}_{\delta}^{h}\|_{\infty}\leq\|(\mathcal{L}_{\delta}^{h})^{-1}\|_{\infty}\|(\mathcal{L}_{\delta}\mathbf{u}_{\delta}-\mathcal{L}_{\delta}^{h}\mathbf{u}_{\delta})\|_{\infty}\sim\mathcal{O}\left(\frac{h^{2}}{\delta^{2}}\right).
	\end{align}
\end{proof}
The derived error estimate explicitly reveals the interplay between the discretization parameter $h$ and the nonlocal horizon $\delta$. These results indicate that, although the IPA-AC method achieves second-order convergence with respect to the mesh size $h$ for a fixed nonlocal model, it lacks the asymptotic compatibility required to recover the classical local model in the concurrent limit without careful calibration of the $\delta/h$ ratio. In the next subsection, it can be seen that the tensor kernel yields the same results.

\subsection{Convergence analysis of tensor kernel}\label{sec:Convergence tensor kernel}
The theoretical framework established for the scalar kernel can be naturally extended to the tensor kernel model. Although the tensor formulation introduces directional coupling between displacement components, the discrete operator preserves the fundamental symmetry properties inherent to the bond-based peridynamic theory. Consequently, the error decomposition follows the same logic as \cref{eq:errordecomposition}, relying on the stability of the stiffness matrix and the consistency of the truncation error.

For the stability analysis, the structural properties of the stiffness matrix associated with the tensor kernel are analogous to those of the scalar case. Provided that the horizon $\delta$ is sufficiently resolved by the mesh to ensure the connectivity of the interaction graph, the discrete tensor operator $\mathcal{L}_{2,\delta}^{h}$ is invertible.

\begin{lemma}\label{Uniform Boundedness of Tensor kernel}
	(Uniform Boundedness of Tensor kernel)
	The inverse of the discrete tensor operator is uniformly bounded independent of the mesh size $h$ and horizon size $\delta$. That is, there exists a constant $C_{t1}$, independent of $h$ such that
	\begin{equation*} \lVert(\mathcal{L}_{2,\delta}^{h})^{-1}\rVert_{\infty} \leq C_{t1}. 
	\end{equation*}
\end{lemma}
\begin{proof}
	The proof of uniform boundedness proceeds in an manner identical to that of the scalar kernel presented in \cref{Uniform Boundedness of Scalar kernel}. Specifically, the function $\mathbf{v}_{h}$ defined in \cref{Vh} remains valid for the tensor formulation, satisfying the condition
	\begin{equation*} 
		\mathcal{L}_{2,\delta}^{h}\mathbf{v}_{h} = \boldsymbol{1}. 
	\end{equation*} 
	Furthermore, a direct evaluation yields the specific stability constant for the tensor kernel
	\begin{equation*} 
		C_{t1} = c_{2}(\delta)\frac{\pi\delta^{3}}{6}. \end{equation*} 
	Since $c_{2}(\delta)$ scales with $\delta^{-3}$, $C_{t1}$ remains a constant independent of the horizon size, thereby ensuring the uniform boundedness of the inverse operator.
\end{proof}

\begin{lemma}\label{Consistency of Tensor Operator}
(Consistency of Tensor Operator) 
Assume that the exact solution $\mathbf{u}\in C^{2}(\mathcal{B})$ and the tensor kernel $\mathbf{T}(\boldsymbol{\xi})$ is defined as in \cref{tensor_kernel}. The local truncation error of the discrete tensor operator $\mathcal{L}_{2,\delta}^{h}$ obtained via the IPA-AC method is uniformly bounded by \begin{equation*} 
\lVert\mathcal{L}_{2,\delta}\mathbf{u}_{\delta} - \mathcal{L}_{2,\delta}^{h}\mathbf{u}{\delta}\rVert_{\infty} \leq C_{t2} \frac{h^2}{\delta^2},
\end{equation*} 
where $C_{t2}$is a generic constant independent of the mesh size $h$ and the horizon $\delta$.
\end{lemma}

\begin{proof}
The proof of consistency for the tensor operator follows a procedure analogous to that of the scalar kernel presented in \cref{Consistency of Scalar Operator}. The derivation remains structurally identical, with the primary distinction arising in the expansion of the leading-order term in Part A , corresponding to \cref{different}. For the sake of brevity, we omit the redundant steps and explicitly derive only this specific term as follows
\begin{equation*}
\begin{aligned}
A&=\sum_{p}[\partial_{i}\mathbf{T}(\mathbf{x},\mathbf{y}^{cp})\partial_{j,k}^{2}u_{a}(\mathbf{x})\Delta x_{k}+\frac{1}{4}\partial_{i,j}^{2}\mathbf{T}(\mathbf{x},\mathbf{y}^{cp})\partial_{k,l}^{2}u_{a}(\mathbf{x})\Delta x_{k}\Delta x_{l}]\int_{\Omega_{p}}\Delta y_{i}^{cp}\Delta y_{j}^{cp}d\mathbf{y}\\
&=\mathcal{O}\left(\frac{\delta^{2}}{h^{2}}[\delta^{-5}\cdot\delta+\delta^{-6}\cdot\delta^{2}]h^{4}\right)\\
&=\mathcal{O}\left(\frac{\delta^{2}}{h^{2}}\right)
\end{aligned}
\end{equation*}
\end{proof}

\begin{theorem}\label{Convergence of Tensor Kernel}
(Convergence of Tensor Kernel) Let $\mathbf{u}_{\delta}$ be the solution to the continuous peridynamic equation with the tensor kernel, and $\mathbf{u}_{\delta}^{h}$ be the corresponding discrete solution. The numerical scheme converges in the infty norm with the following error estimate: 
\begin{equation*} 
	\lVert\mathbf{u}_{\delta} - \mathbf{u}_{\delta}^{h}\rVert_{\infty}
	\sim \mathcal{O}\left(\frac{h^2}{\delta^2}\right). 
\end{equation*}
\end{theorem}
\begin{proof}
Based on the consistency and stability results established in \cref{Uniform Boundedness of Tensor kernel} and \cref{Consistency of Tensor Operator},
$$\lVert\mathbf{u}_{\delta} - \mathbf{u}_{\delta}^{h}\rVert_{\infty} 
\leq \lVert(\mathcal{L}_{2,\delta}^{h})^{-1}\rVert_{\infty} \lVert\boldsymbol{\tau}\rVert_{\infty} 
\leq C_{t2} \cdot C_{t1} \frac{h^2}{\delta^2} 
= \mathcal{O}\left(\frac{h^2}{\delta^2}\right).$$
\end{proof}
The analytical results for the tensor kernel confirm the universality of the convergence behaviors observed in the scalar case. Specifically, the IPA-AC method exhibits second-order convergence in $h$ for a fixed horizon, but suffers from error amplification as $\mathcal{O}(\delta^{-2})$ when the horizon decreases with a fixed mesh. This reinforces the conclusion that the method, while accurate for fixed nonlocal scales, does not satisfy the asymptotic compatibility.

\section{Numerical Examples}\label{Sec:Numerical}
In this section, we present a comprehensive set of numerical experiments to substantiate the theoretical convergence analysis established in \cref{Sec:Convergence}. The primary objective is to validate the error estimates derived for the IPA-AC algorithm, specifically confirming the predictions of \cref{Convergence of Scalar Kernel} and \cref{Convergence of Tensor Kernel} under three distinct limiting regimes:

$h$-convergence: Evaluation of convergence with respect to mesh refinement for a fixed horizon $\delta$. This test is designed to validate the theoretical second-order convergence rate $\mathcal{O}(h^{2})$.

$\delta$-convergence: Investigation of the influence of the nonlocal horizon $\delta$ on the discretization error for a fixed mesh size $h$. This test validates the negative second-order dependency $\mathcal{O}(\delta^{−2})$ identified in the error bounds.

Asymptotic Compatibility test: Assessment of the error behavior when the ratio $m=\delta/h$ is held constant while $(h,\delta)\rightarrow0$. We aim to demonstrate that the convergence order in this regime is zeroth-order $\mathcal{O}(1)$, indicating error stagnation and a lack of asymptotic compatibility.

All simulations are conducted on a two-dimensional square domain $\mathcal{B}=[0,1]\times[0,1]$. To rigorously quantify the discretization error, the exact solution $\mathbf{u}(\mathbf{x})$ is prescribed. The corresponding body force density $\mathbf{b}(\mathbf{x})$ is derived analytically by directly applying the continuous peridynamic operator to the exact solution. The numerical accuracy is assessed using the discrete $L_{\infty}$ norm:
\begin{equation*}
	\lVert\mathbf{e}\rVert_{\infty} = \max_{\mathbf{x}_i \in \mathcal{L}_{\mathcal{B}}} \lVert\mathbf{u}_{\delta}(\mathbf{x}_i) - \mathbf{u}_{\delta}^{h}(\mathbf{x}_i)\rVert_{\infty}.
\end{equation*}

\subsection{Numerical examples of scalar kernel}

For the scalar kernel defined in \cref{scalar_kernel}, the nonlocal interactions are isotropic and decoupled. Consequently, the vector-valued problem reduces to a system of independent scalar equations. Without loss of generality, we treat the solution $u(\mathbf{x})$ as a scalar field.

We consider three test cases with increasing complexity:

\begin{itemize}
	\item \textbf{Case 1 (Quadratic):} 
	$$u_{1}(\mathbf{x}) = \frac{x_{1}(1 - x_{1})}{2} + \frac{x_{2}(1 - x_{2})}{2}.$$
	This function vanishes on the boundary and represents the function constructed in our stability analysis.
	
	\item \textbf{Case 2 (Cubic/Quadratic Mix):} 
	$$u_{2}(\mathbf{x}) = x_{1}^{3} + 2x_{2}^{2}.$$
	This case introduces asymmetry and higher-order polynomial terms.
	
	\item \textbf{Case 3 (High-order Coupled Polynomial):} 
	$$u_{3}(\mathbf{x}) = x_{1}^{3} x_{2}^{2} + x_{2}^{4}.$$
	This function tests the method's capability to handle complex, coupled polynomial variations.
\end{itemize}

\subsubsection{Convergence with fixed horizon}\label{Convergence with fixed horizon}
In this experiment, we verify the theoretical error estimate with respect to the mesh size $h$ in \cref{Convergence of Scalar Kernel}. We fix the horizon size at $\delta=0.4$ and refine the grid spacing h successively from $h=0.2~(\delta/2)$ down to $h=0.0125~(\delta/32)$.

\cref{tab:scalar_fixed_delta} summarizes the $L_{\infty}$ errors and the corresponding convergence rates for all three test cases.

Case 1 (Quadratic): The error decreases from $3.70\times10^{-2}$ to $1.68\times10^{-4}$. The computed convergence rate starts at $1.86$ and steadily approaches $2.00$, demonstrating a robust second-order convergence.

Case 2 and Case 3 (Higher-order): Both cases exhibit consistent error decay behaviors. For Case 2, the rate evolves from $1.83$ to $2.00$; similarly, Case 3 shows an improvement from $1.85$ to $2.00$.

As visually corroborated in \cref{fig:h-convergence}, the convergence slopes for all cases asymptotically align with the theoretical reference line of slope $2$. These numerical results rigorously confirm that the IPA-AC method achieves second-order accuracy $\mathcal{O}(h^{2})$ for fixed nonlocal models, effectively eliminating the first-order truncation errors typically associated with boundary effects.
\begin{figure}
	\centering
	\includegraphics[width=0.7\linewidth]{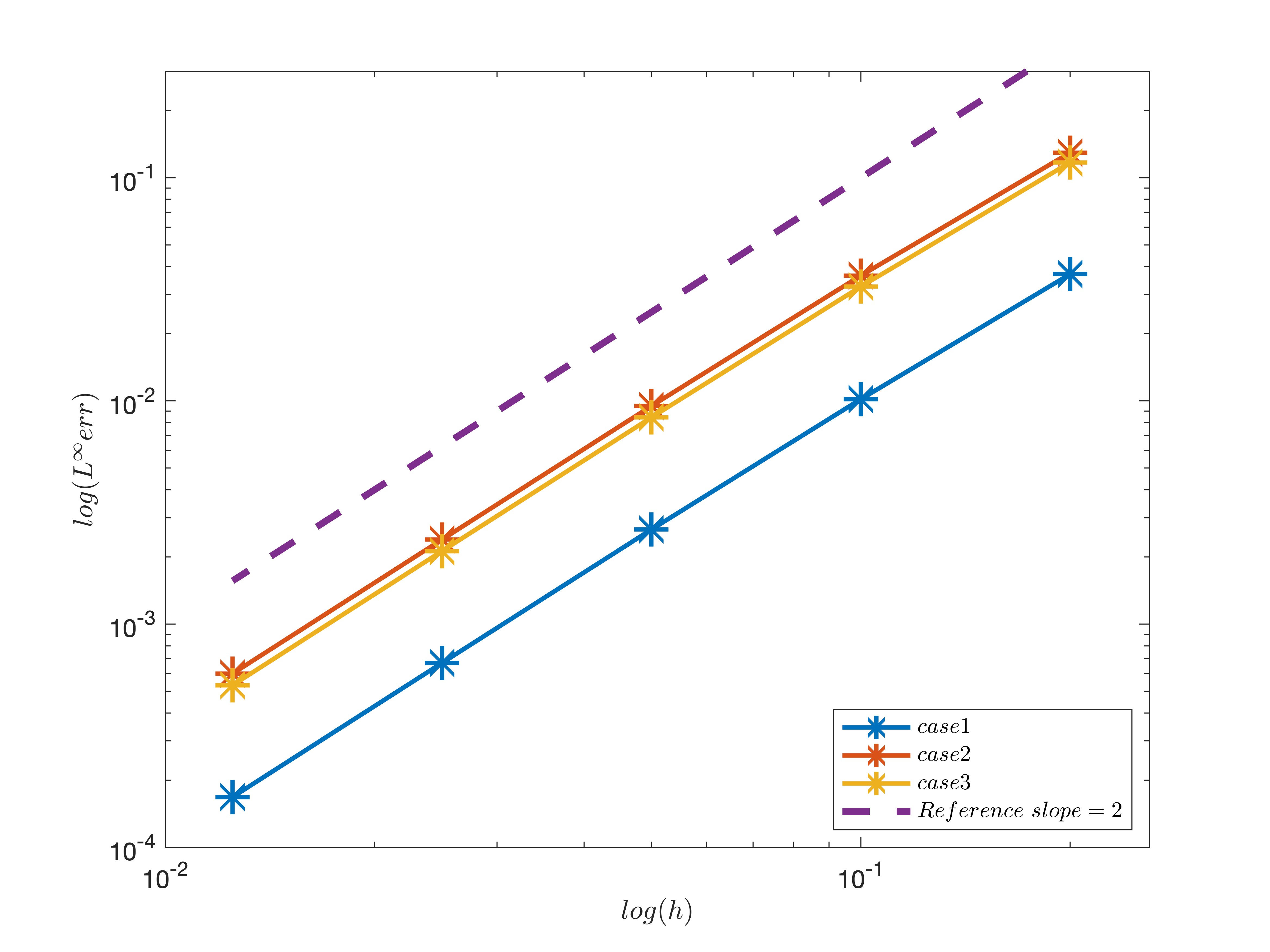}
	\caption{The $h$-convergence of the discretization error for a fixed horizon $\delta$. The solid lines with symbols represent numerical results for Cases 1–3, while the dashed reference line indicates the theoretical $\mathcal{O}(h^{2})$ convergence rate.}
	\label{fig:h-convergence}
\end{figure}

\begin{table}[htbp]
	\centering
	\caption{Convergence results for the scalar kernel with fixed horizon $\delta = 0.4$.}
	\label{tab:scalar_fixed_delta}
	\begin{tabular}{|c|cc|cc|cc|}
		\hline
		\multirow{2}{*}{$h$} & \multicolumn{2}{|c|}{Case 1} & \multicolumn{2}{c|}{Case 2} & \multicolumn{2}{c|}{Case 3} \\ \cline{2-7} 
		& $\|e\|_{\infty}$ & Order & $\|e\|_{\infty}$ & Order & $\|e\|_{\infty}$ & Order \\ \hline
		0.20000 & 3.70e-02 & - & 1.30e-01 & - & 1.17e-01 & - \\
		0.10000 & 1.02e-02 & 1.86 & 3.64e-02 & 1.83 & 3.25e-02 & 1.85 \\
		0.05000 & 2.65e-03 & 1.94 & 9.49e-03 & 1.94 & 8.43e-03 & 1.95 \\
		0.02500 & 6.70e-04 & 1.99 & 2.39e-03 & 1.99 & 2.12e-03 & 1.99 \\
		0.01250 & 1.68e-04 & 2.00 & 6.00e-04 & 2.00 & 5.32e-04 & 2.00 \\ \hline
	\end{tabular}
\end{table}

\subsubsection{Influence of nonlocal horizon}
In this set of experiments, we isolate the influence of the horizon size $\delta$ on the discretization error. We fix the mesh size at $h=0.01$ and gradually decrease $\delta$ from $0.1$ to $0.03$. This corresponds to a reduction in the horizon-to-mesh ratio m from $10$ down to $3$.

\cref{tab:scalar_fixed_h} presents the error evolution and the computed convergence rates with respect to $\delta$.

Across all three test cases, the discretization error exhibits a significant increase as the horizon shrinks, as illustrated by the log-log plot in \cref{fig:deltacovergence}. For instance, in Case 1, the error amplifies by nearly an order of magnitude from $1.32\times10^{-3}$ to $1.27\times10^{-2}$ as $\delta$ decreases.

The computed rates for all cases consistently settle in the range of $-1.85$ to $-1.92$. This strong negative correlation provides empirical verification of the theoretical truncation error bound derived in \cref{Consistency of Scalar Operator}, which predicts a dependency of $\mathcal{O}(\delta^{−2})$.

The results confirm that for a fixed grid resolution, the accuracy of the IPA-AC method is sensitive to the nonlocal length scale. The negative second-order convergence trend highlights the necessity of maintaining a sufficient horizon size relative to the mesh to suppress the $\mathcal{O}(h^{2}/\delta^{2})$ term in the error estimate.
\begin{figure}[htbp]
	\centering
	\includegraphics[width=0.7\linewidth]{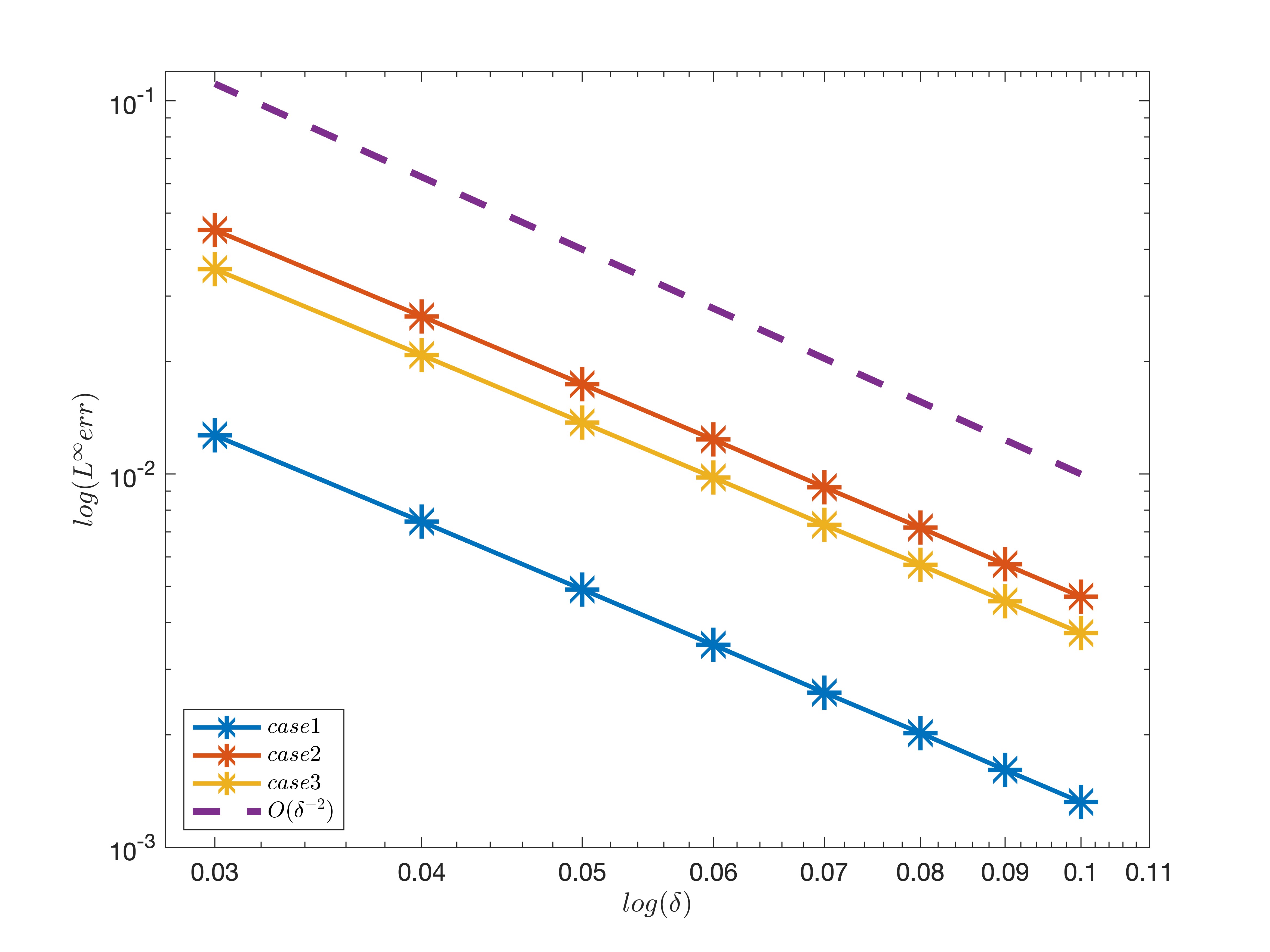}
	\caption{The $\delta$-convergence of the error for a fixed $h$ . The solid lines with symbols represent numerical results for Cases 1–3, while the dashed reference line indicates the theoretical $\mathcal{O}(\delta^{2})$ convergence rate.}
	\label{fig:deltacovergence}
\end{figure}

\subsubsection{Asymptotic Compatibility test}
Finally, we assess the asymptotic compatibility of the proposed scheme. This test involves refining both the mesh size $h$ and the horizon $\delta$ simultaneously towards zero, while maintaining a fixed horizon-to-mesh ratio $m=\delta/h$. We consider three distinct ratios, $m=[3,4,5]$, to investigate the sensitivity of the error to the relative resolution.

\cref{tab:ac_test} details the $L_{\infty}$ errors and the corresponding convergence rates for Cases 1, 2, and 3 under different $m$ ratios. As shown in \cref{fig:mcovergence}, the results reveal a consistent trend across all test functions.

As the discretization is refined ($h\rightarrow0$), the computed convergence orders rapidly decay towards zero. For instance, in Case 1 with $m=3$, the rate drops from $0.19$ to $0.02$; similarly, in Case 3 with $m=5$, despite starting higher at $0.45$, the rate eventually diminishes to $0.05$.

The discretization error does not converge to zero but rather stagnates at a finite, non-zero value. For Case 2 $m=4$, the error levels off at approximately $2.6\times10^{-2}$, indicating an $\mathcal{O}(1)$ scaling behavior.
\begin{table}[htbp]
	\centering
	\caption{Error analysis for the scalar kernel with fixed mesh size $h = 0.01$ and varying horizon $\delta$.}
	\label{tab:scalar_fixed_h}
	\begin{tabular}{|c|cc|cc|cc|}
		\hline
		\multirow{2}{*}{$\delta$} & \multicolumn{2}{|c|}{Case 1} & \multicolumn{2}{c|}{Case 2} & \multicolumn{2}{c|}{Case 3} \\ \cline{2-7} 
		& $\|e\|_{\infty}$ & Order & $\|e\|_{\infty}$ & Order & $\|e\|_{\infty}$ & Order \\ \hline
		0.10 & 1.32e-03 & - & 4.69e-03 & - & 3.75e-03 & - \\
		0.09 & 1.61e-03 & -1.89 & 5.73e-03 & -1.89 & 4.56e-03 & -1.86 \\
		0.08 & 2.02e-03 & -1.93 & 7.18e-03 & -1.92 & 5.71e-03 & -1.90 \\
		0.07 & 2.60e-03 & -1.87 & 9.22e-03 & -1.87 & 7.31e-03 & -1.85 \\
		0.06 & 3.49e-03 & -1.91 & 1.24e-02 & -1.91 & 9.79e-03 & -1.89 \\
		0.05 & 4.91e-03 & -1.88 & 1.74e-02 & -1.88 & 1.37e-02 & -1.86 \\
		0.04 & 7.45e-03 & -1.87 & 2.65e-02 & -1.87 & 2.08e-02 & -1.86 \\
		0.03 & 1.27e-02 & -1.85 & 4.51e-02 & -1.85 & 3.54e-02 & -1.84 \\ \hline
	\end{tabular}
\end{table}
While increasing the ratio m reduces the magnitude of the residual error (e.g., in Case 1, the final error drops from $1.25\times10^{-2}$ at $m=3$ to $4.82\times10^{-3}$
at $m=5$), it does not alter the fundamental lack of convergence. In the log-log plot \cref{fig:mcovergence}, this is manifested as a vertical downward shift of the stagnation lines as m increases.

These numerical observations provide conclusive evidence that the IPA-AC method is not asymptotically compatible. The fixed ratio m effectively locks the truncation error at a constant level, preventing the discrete nonlocal operator from recovering the local limit solution in the concurrent limit. Consequently, to achieve high accuracy in practice, one must either employ a fixed nonlocal model as shown in \cref{Convergence with fixed horizon} or ensure that the horizon converges much slower than the mesh size.
\begin{figure}
	\centering
	\includegraphics[width=0.7\linewidth]{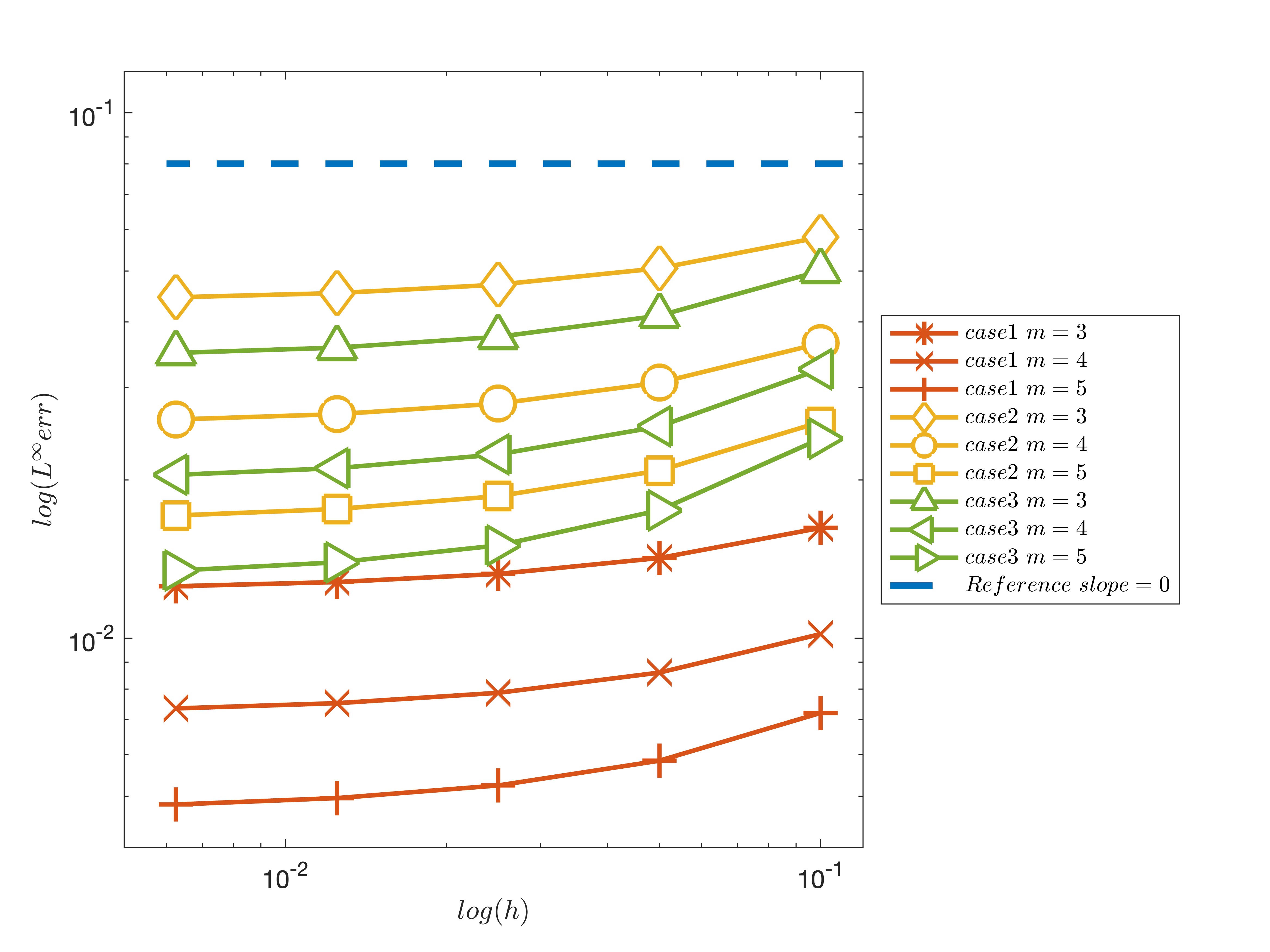}
	\caption{Asymptotic compatibility test with fixed horizon-to-mesh ratios $m\in\{3,4,5\}$. The stagnation of error curves (slope$=0$) as $h\rightarrow0$ demonstrates the lack of asymptotic compatibility, while the downward shift of the plateaus for larger $m$ reflects the reduction in modeling error.}
	\label{fig:mcovergence}
\end{figure}

\begin{table}[htbp]
	\centering
	\caption{Asymptotic compatibility test results: Convergence behavior under fixed ratios $m=\delta/h$ with refining $h$.}
	\label{tab:ac_test}
	\begin{tabular}{|c|c|cc|cc|cc|}
		\hline
		\multirow{2}{*}{Ratio} & \multirow{2}{*}{$h$} & \multicolumn{2}{|c|}{Case 1} & \multicolumn{2}{c|}{Case 2} & \multicolumn{2}{c|}{Case 3} \\ \cline{3-8} 
		&  & $\|e\|_{\infty}$ & Order & $\|e\|_{\infty}$ & Order & $\|e\|_{\infty}$ & Order \\ \hline
		\multirow{5}{*}{$m=3$} 
		& 0.1000 & 1.62e-02 & - & 5.80e-02 & - & 5.00e-02 & - \\
		& 0.0500 & 1.42e-02 & 0.19 & 5.06e-02 & 0.19 & 4.11e-02 & 0.28 \\
		& 0.0250 & 1.32e-02 & 0.10 & 4.70e-02 & 0.10 & 3.74e-02 & 0.13 \\
		& 0.0125 & 1.27e-02 & 0.05 & 4.53e-02 & 0.05 & 3.57e-02 & 0.06 \\
		& 0.0062 & 1.25e-02 & 0.02 & 4.45e-02 & 0.02 & 3.49e-02 & 0.03 \\ \hline
		\multirow{5}{*}{$m=4$} 
		& 0.1000 & 1.01e-02 & - & 3.64e-02 & - & 3.25e-02 & - \\
		& 0.0500 & 8.60e-03 & 0.24 & 3.06e-02 & 0.24 & 2.52e-02 & 0.36 \\
		& 0.0250 & 7.87e-03 & 0.12 & 2.79e-02 & 0.13 & 2.23e-02 & 0.17 \\
		& 0.0125 & 7.52e-03 & 0.06 & 2.67e-02 & 0.06 & 2.10e-02 & 0.08 \\
		& 0.0062 & 7.35e-03 & 0.03 & 2.60e-02 & 0.03 & 2.04e-02 & 0.04 \\ \hline
		\multirow{5}{*}{$m=5$} 
		& 0.1000 & 7.20e-03 & - & 2.58e-02 & - & 2.39e-02 & - \\
		& 0.0500 & 5.84e-03 & 0.30 & 2.08e-02 & 0.30 & 1.75e-02 & 0.45 \\
		& 0.0250 & 5.24e-03 & 0.15 & 1.86e-02 & 0.16 & 1.50e-02 & 0.22 \\
		& 0.0125 & 4.96e-03 & 0.07 & 1.76e-02 & 0.08 & 1.39e-02 & 0.10 \\
		& 0.0062 & 4.82e-03 & 0.04 & 1.71e-02 & 0.04 & 1.34e-02 & 0.05 \\ \hline
	\end{tabular}
\end{table}
\subsection{Numerical examples of tensor kernel}
The consistency of these results with the scalar case confirms that the lack of asymptotic compatibility is an intrinsic property of the IPA-AC method, independent of the kernel structure. A comprehensive summary of these findings and their implications for practical applications will be presented in the concluding section.

We consider a displacement field where each component follows the quadratic profile used in the scalar analysis
\begin{equation*} 
\mathbf{u}(\mathbf{x}) =
\begin{pmatrix} v(\mathbf{x}) \ v(\mathbf{x}) \end{pmatrix}, \quad \text{with } v(\mathbf{x}) 
= \frac{x_1(1-x_1)}{2} + \frac{x_2(1-x_2)}{2}. 
\end{equation*} 
Although the exact solution components are identical, the discrete operator $L_{2,\delta}^{h}$ couples them through the shear-related terms in the tensor kernel. We perform the same set of three limiting tests to verify the theoretical predictions of \cref{Convergence of Tensor Kernel}.

\subsubsection{Convergence with fixed horizon}
We first verify the convergence rate with respect to the mesh size $h$ by fixing the horizon at $\delta=0.4$. The mesh is refined from $h=0.2~(m=2)$ down to $h=0.0125~(m=32)$.

\cref{tab:tensor_fixed_delta} summarizes the results. The $L_{\infty}$ error decays from $2.44\times10^{-2}$ to $8.61\times10^{-5}$. The computed convergence rates are remarkably stable, consistently hovering around $2.03$ to $2.04$, which slightly exceeds the theoretical expectation of $2$. This robust second-order convergence confirms that the symmetry-based error cancellation mechanism remains effective for the coupled tensor system, validating the bound $\lVert e\rVert_{\infty}\sim\mathcal{O}(h^{2})$ for fixed horizons.
\begin{table}[htbp]
	\centering
	\caption{Convergence results for the tensor kernel with fixed horizon $\delta = 0.4$.}
	\label{tab:tensor_fixed_delta}
	\begin{tabular}{|c|cc|}
		\hline
		$h$ & $\|e\|_{\infty}$ & Order \\ \hline
		0.20000 & 2.44e-02 & - \\
		0.10000 & 5.92e-03 & 2.04 \\
		0.05000 & 1.44e-03 & 2.04 \\
		0.02500 & 3.51e-04 & 2.04 \\
		0.01250 & 8.61e-05 & 2.03 \\ \hline
	\end{tabular}
\end{table}
\subsubsection{Influence of nonlocal horizon}
Next, we investigate the sensitivity of the error to the nonlocal horizon $\delta$ under a fixed mesh resolution $h=0.01$. The horizon is decreased from $0.1$ to $0.03$, corresponding to a reduction in the ratio $m$ from $10$ to $3$.

The results in \cref{tab:tensor_fixed_h} reveal a clear trend: the error increases as the horizon shrinks. The computed convergence orders with respect to $\delta$ range from $-1.94$ to $-1.98$. This strong negative correlation provides empirical confirmation of the theoretical estimate $\mathcal{O}(\delta^{-2})$. It indicates that even for the tensor kernel, the discretization error is inversely proportional to the square of the horizon size, highlighting the risk of accuracy loss when $\delta$ approaches the grid scale.

\begin{table}[htbp]
	\centering
	\caption{Error analysis for the tensor kernel with fixed mesh size $h = 0.01$ and varying horizon $\delta$.}
	\label{tab:tensor_fixed_h}
	\begin{tabular}{|c|cc|}
		\hline
		$\delta$ & $\|e\|_{\infty}$ & Order \\ \hline
		0.10 & 6.85e-04 & - \\
		0.09 & 8.41e-04 & -1.94 \\
		0.08 & 1.06e-03 & -1.96 \\
		0.07 & 1.37e-03 & -1.94 \\
		0.06 & 1.86e-03 & -1.98 \\
		0.05 & 2.66e-03 & -1.96 \\
		0.04 & 4.14e-03 & -1.98 \\
		0.03 & 7.30e-03 & -1.97 \\ \hline
	\end{tabular}
\end{table}

\subsubsection{Asymptotic Compatibility test}
Finally, we test the asymptotic compatibility by refining both $h$ and $\delta$ while keeping their ratio $m=\delta/h$ constant at $m=[3,4,5]$.

\cref{tab:tensor_ac_test} presents the convergence behavior in the fixed ratio $m=\delta/h$. Mirroring the trends identified in the scalar analysis, the numerical results exhibit the following characteristics:

The errors do not vanish as $h\rightarrow0$. For instance, with $m=5$, the error stabilizes at approximately $2.61\times10^{-3}$.

The convergence orders drop rapidly towards zero (e.g., reaching as low as $0.03$ for $m=3$ and $0.05$ for $m=5$).

Increasing $m$ reduces the absolute error magnitude (from $7.21\times10^{-3}$ at $m=3$ to $2.61\times10^{-3}$ at $m=5$), but fails to recover the convergence rate.

These results conclusively demonstrate that the IPA-AC method applied to the tensor kernel is not asymptotically compatible. The error behavior is dominated by the path-dependent term $\mathcal{O}(h/\delta)^{2})=\mathcal{O}(m^{−2})$, which remains constant in this scaling regime.
\begin{table}[htbp]
	\centering
	\caption{Asymptotic compatibility test results for the tensor kernel: Convergence behavior under fixed ratios $m=\delta/h$.}
	\label{tab:tensor_ac_test}
		\begin{tabular}{|c|cc|cc|cc|}
			\hline
			\multirow{2}{*}{$h$} & \multicolumn{2}{|c|}{$m=3$} & \multicolumn{2}{c|}{$m=4$} & \multicolumn{2}{c|}{$m=5$} \\ \cline{2-7} 
			& $\|e\|_{\infty}$ & Order & $\|e\|_{\infty}$ & Order & $\|e\|_{\infty}$ & Order \\ \hline
			0.10000 & 9.66e-03 & - & 5.92e-03 & - & 4.13e-03 & - \\
			0.05000 & 8.31e-03 & 0.22 & 4.88e-03 & 0.28 & 3.25e-03 & 0.34 \\
			0.02500 & 7.67e-03 & 0.11 & 4.41e-03 & 0.15 & 2.87e-03 & 0.18 \\
			0.01250 & 7.37e-03 & 0.06 & 4.19e-03 & 0.08 & 2.70e-03 & 0.09 \\
			0.00625 & 7.21e-03 & 0.03 & 4.08e-03 & 0.04 & 2.61e-03 & 0.05 \\ \hline
		\end{tabular}
\end{table}

\section{Conclusion}\label{Sec:Conclusion}
In this work, we have presented a comprehensive theoretical and numerical analysis of the IPA-AC method for peridynamic models. By establishing a rigorous error estimate framework based on the Lax Equivalence Theorem, we derived the explicit convergence rates for both scalar and tensor kernels under various limiting regimes.

The theoretical and numerical results consistently reveal the dual nature of the IPA-AC method:
When the horizon $\delta$ is fixed as a physical parameter, the method exhibits robust second-order convergence $\mathcal{O}(h^{2})$ in the $L_{\infty}$ norm. The geometric corrections effectively eliminate the first-order boundary errors common in standard meshfree discretizations.

However, the method is proven not to be asymptotically compatible. The discretization error scales as $\mathcal{O}(\delta^{-2})$, preventing the recovery of the local limit solution in the concurrent limit $(h,\delta)\rightarrow0$ unless the ratio $m=\delta/h$ is sufficiently large.

These findings offer practical insights into the optimal applicability of the IPA-AC method. The method emerges as a superior candidate for simulating fixed nonlocal models, particularly in scenarios where accurately capturing boundary effects is paramount. In contrast, for problems involving a varying horizon $\delta$, the horizon-to-mesh ratio requires careful calibration to effectively suppress the intrinsic discretization error.

Finally, we remark that the theoretical framework established in this work is naturally extensible to three-dimensional problems. Furthermore, while the current IPA-AC method is not asymptotically compatible, its precise geometric representation offers a valuable foundation for future improvements. By increasing the number of integration points and re-calibrating the quadrature weights to strictly enforce polynomial reproduction conditions, it is feasible to develop an enhanced scheme that retains high geometric fidelity while satisfying the requirements for asymptotic compatibility.




\end{document}